\documentclass{article}
\usepackage{amssymb}
\title{{\bf Splitting via Noncommutativity}}
\author{  {\bf M. L. Lewis}, \  {\bf D. V. Lytkina},  \ {\bf V. D. Mazurov}, \\[0.3cm]  and  {\bf A. R.  Moghaddamfar}
}

\newenvironment{proof}{\noindent {\em {Proof}}.}{$\square$
\medskip}

\newtheorem{theorem}{Theorem}[section]

\newtheorem{corollary}[theorem]{Corollary}

\newtheorem{lm}[theorem]{Lemma}

\begin{document}
\newcommand{\f}{\frac}
\newcommand{\sta}{\stackrel}
\maketitle

\begin{abstract}
\noindent
Let $G$ be a nonabelian group and $n$ a natural number.  We say that $G$ has a strict $n$-split decomposition if it
can be partitioned as the disjoint union of an abelian subgroup $A$ and $n$ nonempty subsets $B_1, B_2, \ldots, B_n$,
such that $|B_i| > 1$ for each $i$ and within each set $B_i$, no two distinct elements commute.
We show that every finite nonabelian group has a strict $n$-split decomposition for some $n$. We classify all finite groups
$G$, up to isomorphism, which have a strict $n$-split decomposition for $n = 1, 2, 3$.  Finally, we show that for a nonabelian
group $G$ having a strict $n$-split decomposition, the index $|G:A|$ is bounded by some function of $n$.
\end{abstract}


\renewcommand{\baselinestretch}{1}
\def\thefootnote{ \ }
\footnotetext{{\em $2010$ Mathematics Subject Classification}:
20D05.}

\section{Introduction}
Throughout this paper, all groups are finite.  We are interested in studying how to split up a nonabelian group into
disjoint subsets where the members of the subsets do not commute with each other.  Let $G$ be a nonabelian
group and $n$ a natural number.  An {\it $n$-split decomposition of $G$} is the disjoint union of an abelian subgroup
$A$ and $n$ nonempty subsets $B_1, B_2, \ldots, B_n$ such that for each $i$ no two distinct elements of the set $B_i$ commute.
When $G$ has an $n$-split decomposition, we will denote the decomposition by $G = A \uplus B_1 \uplus B_2 \uplus \cdots \uplus B_n$.
We will say that the $n$-split decomposition is {\it strict} if $|B_i| > 1$ for each $i$ with $1 \leqslant i \leqslant n$.  When $n = 1$,
we simply say $G = A \uplus B_1$ is a split decomposition of $G$.

We note that groups having split decompositions have previously been studied in \cite{ma},
and this paper arose as a generalization suggested by Prof. Isaacs when reading that paper.
We note that $n$-split decompositions are of interest when consider the commuting graph of a group.
If $G$ is a nonabelian group, then the commuting graph $\Delta (G)$ of $G$  is the graph whose vertex set is
$G \setminus Z(G)$ and there is an edge between $x$ and $y$ if $x$ and $y$ commute.
An $n$-split decomposition of $G$ corresponds to a partition of $\Delta (G)$ where $A \setminus Z(G)$ is a
complete subgraph and the $B_i$'s are disjoint independent subsets.  Saying the $n$-split decomposition is
strict is equivalent to requiring each of the independent subsets to have at least $2$-vertices.

Notice that that every group trivially has a $(|G| - 1)$-decomposition by taking $A = 1$ and taking the $B_i$'s to be
the singleton sets $\{ g \}$ as $g$ runs over the nonidentity elements of $G$.  To eliminate this trivial decomposition,
it is reasonable to ask which groups have a strict $n$-split decomposition for some $n$.  We will prove that {\em all}
nonabelian groups have a strict $n$-split decomposition for some positive integer $n$ and some abelian subgroup $A$.

\begin{theorem}\label{one}
If $G$ is a nonabelian group, then there is an abelian subgroup $A$ and a positive integer $n$ so that $G$ has a strict $n$-split decomposition with respect to $A$.
\end{theorem}

We will see that usually there is no unique (strict) $n$-split decomposition for a given group.
Moreover, we will see that for most groups that have a (strict) $n$-split decomposition will also have a (strict) $(n+1)$-split decomposition.
The question arises, given a group $G$, what is the smallest integer $n$ so that $G$ has a (strict) $n$-split decomposition.
For some classes of groups, we are able to answer this question.

\begin{theorem}\label{two}
For each of the following groups, $n$ is the minimal integer so that $G$ has a (strict) $n$-split decomposition.
\begin{enumerate}
\item If $G$ is a Frobenius group with abelian Frobenius kernel $A$ and an abelian Frobenius complement, then $n = |G:A| - 1$.
\item If $G = L_2 (2^m)$, then $n = 2^m$.
\item If $p$ is an odd prime, $q$ is a power of $p$, and $G = L_2 (q)$, then $n = q-1$.
\item If $p$ is an odd prime, $q$ is a power of $p$, and $G = {\rm PGL}_2 (q)$, then $n = q$.
\item If $G={\rm Sz}(2^{2m+1})$, $m\geqslant 1$, then $n=2^{2m+2}-1$.
\end{enumerate}
\end{theorem}

For small values of $n$, we are able to obtain the structure of groups $G$ which have a strict $n$-split decomposition.
In particular, we present the classification for $n = 1, 2, 3$.  Notice that Theorem \ref{two} (1) shows that it is not possible to bound $|A|$
in terms of $n$ when $G$ has an $n$-split decomposition.  On the other hand, the classifications for $n = 1, 2, 3$
suggest that it may be possible to bound $|G:A|$ in terms of $n$ when $G$ has an $n$-split decomposition.
In our final theorem, we shall show that this is in fact true.

\begin{theorem}\label{three}
There exists a positive integer valued function $f$ defined on the positive integers so that if $G$ has an $n$-split decomposition, then $|G:A| \leqslant f(n)$.
\end{theorem}

\section{The existence of a strict $n$-split decomposition}

In this section, we consider the problem of the existence of a strict $n$-split decomposition for each nonabelian group. We begin with the following observation about when elements commute.

\begin{lm}\label{lm-eee0}
Let $A \subset G$ be an abelian subgroup and fix elements $a_1, a_2 \in A$ and $g \in G$. Then $a_1g$ and $a_2g$ commute if and only if $C_A(g) a_1 = C_A(g) a_2$.
\end{lm}

\begin{proof} This holds by direct calculations. In fact, we have
$$\begin{array}{lll}
a_1 g a_2 g = a_2 g a_1 g & \Longleftrightarrow & a_1 g a_2 = a_2 g a_1 \\[0.2cm]
& \Longleftrightarrow &  a_2^{-1} a_1 g a_2 a_1^{-1} = g \\[0.2cm]
& \Longleftrightarrow & (a_2 a_1^{-1})^{-1} g a_2 a_1^{-1} = g \\[0.2cm]
& \Longleftrightarrow & a_2 a_1^{-1} \in C_A(g) \\[0.2cm]
& \Longleftrightarrow &  C_A(g) a_1 = C_A(g) a_2.
\end{array}$$
The result follows. \end{proof}

We next show how to partition the cosets of an abelian group into noncommuting sets.
A set is {\em noncommuting} if no pair of distinct elements in the set commute. Similarly, a set is {\em commuting} if all
pairs of elements commute.

\begin{lm}\label{lm-eeee0}
Let $A \subset G$ be an abelian subgroup, and fix an element $g\in G$. Let $n = |C_A(g)|$.  Then there exist noncommuting subsets $B_1, \ldots, B_n$ such that $Ag = \uplus_{i=1}^{n} B_i$, with $|B_i| = |A:C_A (g)|$ for all $i = 1, \dots, n$.
\end{lm}

\begin{proof}
We write $C_A (g) = \{c_1, \ldots, c_n\}$. Let $\{a_1, \ldots, a_l\}$ be a transversal for $C_A(g)$ in $A$. So, we have
$$
A = \left\{ c_ia_j \mid i = 1, \ldots, n; \ j=1, \ldots, l \right\}.
$$
For $i=1, \ldots, n$,  set
$$
B_i = \{ c_i a_1 g, \ c_i a_2 g,  \ \ldots, \ c_i a_l g \}.
$$
Then, clearly $Ag = \uplus_{i=1}^{n} B_i$.  Moreover, when, $j_1 \ne j_2$, it follows that $c_i a_{j_1}$ and $c_ia_{j_2}$ lie in different cosets of $C_A(g)$, so $c_ia_{j_1}g$ and $c_ia_{j_2}g$ do not commute by Lemma \ref{lm-eee0}. Therefore, we conclude $B_i$ is a noncommuting set and $|B_i| = l = |A:C_A (g)|$. This completes the proof.
\end{proof}

We are now able to prove that every  nonabelian group has a strict $n$-split decomposition with respect to $A$ for some integer $n$ and some abelian subgroup $A$.

\begin{theorem}\label{th-0}
If $G$ is a nonabelian group, then there exists an abelian subgroup $A$ of $G$ and a positive integer $n$ so that $G$ has a strict $n$-split decomposition with respect to $A$.
\end{theorem}

\begin{proof}
Let $A$ be a maximal abelian subgroup of $G$.  It follows that for every element $g \in G \setminus A$, the group $\langle A, g \rangle$ is not abelian, and so $C_A(g) < A$.  Let $\{g_1, g_2, \ldots, g_m\}$ be a transversal for $A$ in $G$ with $g_1 \in A$. Clearly, for every value $i$ with $1 < i \leqslant m$, we have $|A : C_A (g_i)| > 1$.  Furthermore, by Lemma \ref{lm-eeee0}, for each integer $i$ with $1 < i \leqslant m$, there exist noncommuting subsets $B_{i,1}, \ldots, B_{i, l_i}$ so that $Ag_i = \uplus_{j=1}^{l_i} B_{i, j}$, where $l_i = |C_A (g_i)|$ and $|B_{i,j}| = |A : C_A (g_i)| > 1$.  It follows that $$G = A \uplus \biguplus_{i=2}^{m} \biguplus_{j=1}^{l_i} B_{i,j}.$$  This shows that $G$ has a strict $n$-split decomposition where $n = \sum_{i=2}^{m} l_i$.
\end{proof}

We make several observations here.  First, note that our proof can be modified to say: {\em Let $A$ be a maximal abelian subgroup of  $G$, then there is a positive integer $n$ so that $G$ has a strict $n$-split decomposition with respect to $A$}.  We will see that $G$ may also have strict split decompositions with respect to abelian subgroups that are not maximal.  See for example, Lemma 4.4 (4) where we have a strict $2$-split decomposition for $S_3$ with the trivial abelian subgroup which is definitely not a maximal abelian subgroup. We will also see that there are abelian subgroups for which the group will have no strict split decompositions.  In particular, in Lemma 3.1 (5), we see that no proper subgroup of the center of a group can have a strict $n$-decomposition.

Second, note in the proof of Theorem \ref{th-0} that the $n$-split decomposition we obtain has at least one $B_i$ for each coset of $A$ in $G$.  This shows that $G$ will always have an $n$-split decomposition with $n \geqslant |G:A| - 1$ for every maximal abelian subgroup of $G$.  This shows that $G$ will always have a strict $n$-split decomposition where $|G:A|$ is bounded by $n + 1$.  On the other hand, the $n$-split decompositions found in this theorem likely do not have the smallest $n$ that will work.  We will see that it is possible for $|G:A|$ to be larger than $n+1$.  See Lemma 4.4 (3) and (4) for examples.


\section{Preliminary Results}

If $G$ is a group and $g \in G$, then we write $o (g)$ for the order of $g$.  We start with the following general result:

\begin{lm}\label{lm-e0}
Suppose a group $G$ has an $n$-split decomposition with abelian subgroup $A$.  Then, the following hold:
\begin{enumerate}
\item If $U$ is an abelian subgroup of $G$ not contained in $A$, then $$|U \cap A|(|U:U \cap A| - 1) \leqslant n.$$  In particular, $|U \cap A| \leqslant n$, $|U:U \cap A| \leqslant n + 1$, and $|U| \leqslant 2n$.
\item If $b \in G \setminus A$ with $\langle b \rangle \cap A = 1$, then $(o(b)-1)|C_A(b)| \leqslant n$.
\item If $b \in G \setminus A$, then $|C_A (b)| \leqslant n$ and $o (b) \leqslant 2n$.
\item If $A$ is normal in $G$ and $b \in G \setminus A$, then $o (Ab) \leqslant n + 1$.
\item If the $n$-split decomposition is strict, then $Z(G) \leqslant A$ and $|Z(G)|\leqslant n$.
\end{enumerate}
\end{lm}

\begin{proof}
Let $G = A \uplus B_1 \uplus B_2\uplus \cdots \uplus B_n$ be the $n$-split decomposition of $G$. Suppose that $U$ is an abelian subgroup of $G$ not contained in $A$.  Observe that $U \setminus U \cap A$ is a commuting subset of $G \setminus A$.  Hence, each element of $U \setminus U \cap A$ must lie in a distinct $B_i$.  This implies that $|U \setminus U \cap A| \leqslant n$.  Notice that $|U \setminus U \cap A| = |U| - |U \cap A| = |U \cap A|(|U:U \cap A| - 1)$. This implies that $|U \cap A|(|U:U \cap A| - 1) \leqslant n$.  Since $U$ is not contained in $A$, we see that $(|U:U \cap A| - 1) \geqslant 1$, and so, $|U \cap A| \leqslant n$.  Similarly, as $|U \cap A| \geqslant 1$, we have that $(|U:U \cap A| - 1) \leqslant n$ and so, $|U:U \cap A| \leqslant n + 1$.  We have $|U| = |U \setminus U \cap A| + |U \cap A| \leqslant n + n = 2n$.

Suppose $b \in G\setminus A$.  Observe that $C = C_A (b) \langle b \rangle$ is an abelian subgroup not contained in $A$ and $C \cap A = C_A (b)$.  The results of (2) and (3) now follow from (1).  Note that if $A$ is normal in $G$, then $o (Ab) = |A \langle b \rangle/A| = |\langle b \rangle:\langle b \rangle \cap A|$, so (4) also follows from (1).

If we assume that the $n$-split decomposition is strict, then using the fact that $|B_i|>1$, it follows that $B_i$ does not contain any central element of $G$, and so $Z(G) \subseteq A$.  It follows that $Z(G) \leqslant C_A (b)$ for each $b \in G\setminus A$.  We now have $|Z(G)| \leqslant |C_A (b)| \leqslant n$.
\end{proof}

\begin{lm}\label{lm-e0-d}
Suppose a group $G$ has an $n$-split decomposition with abelian subgroup $A$, $p$ is a prime divisor of $|G|$, and $P$ is a $p$-subgroup of $G$.  If either (1) $p \geqslant n + 2$ or (2) $p^2 > 2n$ and $|P| \geqslant p^2$, then $P \leqslant A$.
\end{lm}
\begin{proof}
Let $b \in G$ be an element of order $p$.  If $b$ is not contained in $A$, then $\langle b \rangle \cap A = 1$.  By Lemma \ref{lm-e0} (2), we have $o(b) \leqslant n + 1$.  Thus, if $p \geqslant n + 2$, then we must have $b \in A$, and if $P$ is a $p$-group, then all of elements lie in $A$, so $P \leqslant A$.  We now assume that $p \leqslant n + 1$ and $p^2 > 2n$ and $|P| \geqslant p^2$.   Suppose that $P$ is not contained in $A$.  Thus, there is an element $b \in P \setminus A$.  By Lemma \ref{lm-e0} (2), we have that $o(b) \leqslant n + 1 \leqslant 2n < p^2$.  This implies that $b$ has order $p$.  If $b \in Z(P)$, then this implies $\langle b \rangle < C_P (b)$.  On the other hand, we know $1 \ne Z(P) \leqslant C_P (b)$, so if $b \not\in Z(P)$, then $\langle b \rangle < C_P (b)$.  Let $c$ be an element of $C_P (b)$ that does not lie in $\langle b \rangle$, and let $U = \langle b, c \rangle$.  Note that $U$ is an abelian subgroup of $G$ that is not contained in $A$ and has order at least $p^2$.  This is a contradiction to Lemma \ref{lm-e0} (1), and hence, $P \leqslant A$.
\end{proof}

We show that the bound obtained in Lemma \ref{lm-e0} (1) cannot be improved.

\begin{lm}
If $G$ is an extra-special group of order $p^3$ for some prime $p$, then $G$ has a (strict) $p(p-1)$-split decomposition and every $n$-split decomposition of $G$ satisfies $n \geqslant p (p-1)$.
\end{lm}

\begin{proof}
Observe that the centralizer $C_G(x)$ of every noncentral element $x$ of $G$ is a maximal subgroup of $G$ of order $p^2$ and is abelian.  In fact, $G$ has $p+1$ abelian subgroups of order $p^2$. Suppose that $G$ has an $n$-split decomposition with abelian subgroup $A$, and let $U$ be an abelian subgroup of order $p^2$ that does not equal $A$.  It follows that either $|U \cap A| = |U:U \cap A| = p$ or $|U \cap A| = 1$ and $|U:U \cap A| = p^2$.  We now apply Lemma \ref{lm-e0} (1).  In the first case, we obtain $n \geqslant |U \cap A|(|U:U \cap A| - 1) = p(p-1)$, and in the second case, we have $n \geqslant |U:U \cap A| - 1 = p^2 - 1$.  In either case, we have $n \geqslant p^2 - p$.

We now prove that $G$ has an $(p^2-p)$-split decomposition.  Let $A$ be an abelian subgroup of order $p^2$, and fix an element $g \in G \setminus A$.  We know that $g, g^2, \dots, g^p$ forms a transversal for $A$ in $G$; moreover,  $g^i \not\in A$ for $1 \leqslant i \leqslant p-1$.  It follows that $C_A (g^i) = Z(G)$ has order $p$.  We now apply Lemma \ref{lm-eeee0} to obtain noncommuting sets $B_{i,1}, \dots, B_{i,p}$ so that $Ag^i = B_{i,1} \uplus \cdots \uplus B_{i,p}$ and note that $|B_{i,j}| = p$.  We now obtain $G = A \uplus B_{1,1} \uplus \cdots \uplus B_{p-1,p}$ is a strict $(p^2-p)$-split decomposition of $G$.
\end{proof}

Here we present some examples of nonabelian groups $G$ having a $n$-split decomposition, for some natural number $n$.  The first example concerns a Frobenius group $G$  with an abelian kernel $A$.

\begin{lm} \label{Frobenius}
If $G$ is a  Frobenius group with abelian Frobenius kernel $A$, then the cosets of $A$ form a strict $(|G:A|-1)$-split decomposition of $G$ with respect to $A$.  If in addition, $G$ has an abelian Frobenius complement, then any $n$-split decomposition of $G$ with respect to $A$ has $n \geqslant |G:A| - 1$.
\end{lm}
\begin{proof}
It is obvious that $G$ is the disjoint union of the cosets of $A$.  If $B$ is a coset of $A$ that does not equal $A$ and $b \in B$, then we know that $C_G (b)$ is contained in a Frobenius complement of $G$, and so, $C_G (b) \cap B = \{ b \}$.  It follows that $b$ does not commute with any other element of $B$.  This proves that the cosets of $A$ form a strict $(|G:A|-1)$-split decomposition of $G$ with respect to $A$.  Suppose $G$ has an $n$-split decomposition with respect to $A$ and that $U$ is an abelian Frobenius complement of $G$.  We know that $|U \cap A| = 1$, so $|U| = |U:U \cap A|$, and by Lemma \ref{lm-e0} (1), we have $|U| \leqslant n + 1$.  Since $|U| = |G:A|$, we conclude that $|G:A| - 1 \leqslant n$.
\end{proof}

In particular, a Frobenius group of order $2m$, where $m$ is odd and the kernel is abelian of order $m$, has a split decomposition over the Frobenius kernel.

The following result is proved in \cite{ma}, and states that these are the only possible split decompositions.  We provide a direct group-theoretic proof for the sake of completeness, which was inspired by Isaacs.

\begin{corollary}\label{th-1}
The following conditions on a nonabelian group $G$ are equivalent:
\begin{itemize}
\item[$(1)$] $G$ has a split decomposition with respect to an abelian subgroup $A$.
\item[$(2)$] $G$ is a Frobenius group of order $2n$, where $n$ is odd, and the Frobenius kernel is abelian of order $n$ and is $A$.
\end{itemize}
\end{corollary}

\begin{proof}
$(1)\Rightarrow (2).$
Suppose $G = A \uplus B$ is a split decomposition of $G$.
Let $b$ in $B$.  Then $(o (b) - 1) |C_A (b)| \leqslant 1$ by Lemma \ref{lm-e0} (2).  This implies that $o (b) = 2$ and $C_A(b) = 1$.  Let $P$ be a Sylow $2$-subgroup of $G$.  If $|P| \geqslant 4$, then by Lemma \ref{lm-e0-d}, we would have $P \leqslant A$ which is a contradiction.  Thus, $|P| = 2$, and so $P \cap A = 1$.  Now, we have that $|A|$ is odd, $|G:A| = 2$, and every element outside $A$ has order $2$.  This implies that $G$ is a Frobenius group with abelian Frobenius kernel $A$.

$(2)\Rightarrow (1).$  This is Lemma \ref{Frobenius} when $|G:A| = 2$.
\end{proof}


\section{Groups having a strict  $2$-split decomposition}
We now want to understand groups with an $n$-split decomposition where $n > 1$.  We begin with the following observation that shows that having an $n$-split decomposition with respect to $A$ in most cases will yield an $(n+1)$-split decomposition for the same group $A$.  The following lemma follows easily from the definition.

\begin{lm}\label{lm-ee0}
Suppose the group $G$ has a strict $n$-split decomposition $G=A \uplus B_1 \uplus B_2\uplus \cdots \uplus B_n$, and at least one of the following occurs:
\begin{itemize}
\item[{\rm (a)}] there exists an integer $i$ with $1 \leqslant i \leqslant n$ such that $|B_i|\geqslant 4$.
\item[{\rm (b)}] there exist integers $1 \leqslant i <j \leqslant n$ with $|B_i|\geqslant 3$ and $|B_j|\geqslant 3$, and two elements $x\in B_i$ and $y\in B_j$ such that $xy\ne yx$.
\end{itemize}
Then $G$ has a strict $(n+1)$-split decomposition.
\end{lm}

Hence, a group usually will have strict $n$-split decompositions for different values of $n$.  Notice that, a group $G$ can have more than one (strict) $n$-split decomposition for a fixed $n$.  In particular, we have the following $2$-split decompositions for $S_3$: $A = \{ 1, (12) \}$, $B_1 = \{ (13), (123) \}$, $B_2 = \{ (23), (132) \}$ and $A = \{ 1 \}$, $B_1 = \{ (12), (13), (123) \}$, $B_2 = \{ (23), (132) \}$.

We now work to classify the groups having a $2$-split decomposition.  Lemma \ref{Frobenius} implies that the Frobenius groups with an abelian kernel of index $3$ have a $2$-split decomposition for the Frobenius kernel.  Applying Lemma \ref{lm-ee0} and Corollary \ref{th-1}, we see that every Frobenius group of order $2n$ with an abelian Frobenius kernel of order $n$ except for $D_6 \cong S_3$, also has a $2$-split decomposition.  Note that $S_3$ cannot have a strict $2$-split decomposition over its Frobenius kernel since there are only three elements outside the Frobenius kernel.  On the other hand, the previous paragraph shows that $S_3$ does have strict $2$-split decompositions over other subgroups $A$.  The next result generalizes the fact that Frobenius groups whose Frobenius kernels have index $2$ and order at least $5$ have a strict $2$-decompositions.

\begin{lm}\label{exist}
Let $G$ be a nonabelian group of order greater than $6$ containing an abelian subgroup $A$ of index $2$.  Then $G$ has a (strict) $2$-split decomposition for $A$ if and only if $|Z(G)| \leqslant 2$.
\end{lm}

\begin{proof}
$(\Rightarrow)$
This follows immediately from Lemma \ref{lm-e0} (5).

$(\Leftarrow)$
Suppose first that $Z(G)=1$. Then $G$ is a Frobenius group with kernel $A$ of odd order $\geqslant 5$.  We know $G$ has a split decomposition by Lemma \ref{Frobenius} and also a $2$-split decomposition by Lemma \ref{lm-ee0}(a).  So suppose that $|Z(G)| = 2$.  Let $t$ be an arbitrary element in $G \setminus A$.  Then $C_A(t) = Z(G)$.  By Lemma \ref{lm-eeee0}, there exist noncommuting sets $B_1$ and $B_2$ so that $At = B_1 \uplus B_2$.  We conclude that $G = A \uplus B_1 \uplus B_2$ is a $2$-split decomposition of $G$.
\end{proof}

In light of Lemma \ref{exist}, dihedral, semidihedral, and (generalized) quaternion groups of order at least $8$ have (strict) $2$-split decompositions with respect to their abelian subgroups of index $2$.  We note that the dihedral group of order $8$ and the quaternion group of order $8$ also have strict $2$-split decompositions over their centers.

\begin{lm} \label{index2}
Let $G$ be a nonabelian group.  Then $G$ has an abelian subgroup of index $2$ and $|Z(G)| \leqslant 2$ if and only if $G$ is the semi-direct product of a nontrivial $2$-group $P$ acting on a group $Q$ with odd order where $P$ has an abelian subgroup $B$ of index $2$ and either (1) $|Z(P)| = 2$ or (2) $|P| \leqslant 4$ and $|Q| > 1$, so that $B$ centralizes $Q$ and any element of $P$ outside of $B$ inverts every element of $Q$.  In this situation, if $|P| \geqslant 8$, then $P$ is dihedral, semi-dihedral, quaternion, or generalized quaternion.
\end{lm}

\begin{proof}
Suppose that $G$ has an abelian subgroup $A$ of index $2$ and $|Z(G)| \leqslant 2$.  Let $P$ be a Sylow $2$-subgroup of $G$ and let $Q$ be the Hall $2$-complement  $O_{2'}(A)$ of $A$, and observe that $Q$ is a normal Hall $2$-complement of $G$.  Thus, $G$ is the semi-direct product of $P$ acting on $Q$.  Let $B = P \cap A$ and observe that $B$ is an abelian subgroup of index $2$ in $P$ and centralizes $Q$.  Obviously, $|P| \geqslant 2$.  If $|P| \leqslant 4$, then since $G$ is nonabelian, we must have $|Q|>1$.  Thus, we may assume that $|P| \geqslant 8$.  Notice that $|Z (P)| \geqslant 2$.  Observe that both $P$ and $Q$ will centralize $B \cap Z(P)$, so $B \cap Z(P)$ is central in $PQ = G$, and thus, $B \cap Z(P) \leqslant Z(G)$, and so, $|B \cap Z(P)| \leqslant 2$.  Notice that if $Z (P)$ is not contained in $B$, then $P = B Z(P)$ and both $B$ and $Z(P)$ centralize $B$; so $B$ is central in $P$.  This implies that $B = B \cap Z(P)$.  We deduce that $|B| \leqslant 2$ and $|P| \leqslant 4$ which contradicts $|P| \geqslant 8$.  Thus, $|Z (P)| = |B \cap Z(P)| \leqslant 2$, and conclude that $|Z(P)| = 2$.  Suppose $x$ is an element of $P$ that lies outside $B$.  We know that $Z (G) \leqslant C_G (x)$.  Observe that $x$, $B$, and $Q$ will all centralize $C_Q (x)$, and so, $C_Q (x) \leqslant Z(G) \cap Q$, and thus, $C_Q (x) = 1$.  This implies $x B$ acts Frobeniusly on $Q$, and this implies that $x$ inverts every element of $Q$.

Suppose now that $|P| \geqslant 8$.  Since $|Z(P)| = 2$, it follows that $P$ is a nonabelian group.  If $|P| = 8$, this implies that $P$ is either the dihedral group or the quaternion group.  Thus, we may assume that $|P| \geqslant 16$.  This implies that $|B| \geqslant 8$.  Let $U = \{ b \in B \mid b^4 = 1 \}$.  It is not difficult to see that $U$ is a characteristic subgroup of $B$.  By the Fundamental Theorem of abelian groups, we see that $|U| \geqslant 4$ and $|U| = 4$ if and only if $B$ is cyclic.  Suppose $B$ is not cyclic, so $|U| \geqslant 8$.  Fix an element $x \in P \setminus B$.  Observe that $\langle x \rangle \cap B$ and $C_B (x)$ will both be centralized by both $x$ and $B$ and so are central in $P$.  This implies that $|\langle x \rangle \cap B| \leqslant 2$ and so, $x$ has order at most $4$.  Also, $|C_B (x)| = 2$, so $C_B (x) = Z(P)$. We can find an $x$-invariant subgroup $V$ of $U$ with $Z(P) \leqslant V$ and $|V| = 8$.  Let $D = \langle x \rangle V$, and observe that $|D| = 16$.  Notice that $C_V (x) = C_B(x) \cap V = Z(P) \cap V = Z(P)$.  This implies that $Z(D) = Z(P)$.  We deduce that $D' > 1$, and so $Z (D) \leqslant D'$.  If $D' = Z(D)$, then $D$ would be extra-special which is a contradiction since $|D : Z(D)| = 8$.  Thus, we must have $Z (D) < D'$.  Since $D$ is not cyclic, we know that $|D:D'| \geqslant 4$, and we conclude that $|D:D'| = 4$.  This implies that $D$ is either dihedral, semidihedral, or generalized quaternion.  But in this case, the only abelian subgroup of index $2$ is cyclic since $|D| = 16$.  We deduce that $V$ is cyclic which is a contradiction.  Therefore, $B$ is cyclic, and hence, $P$ is dihedral, semidihedral, or generalized quaternion (see Satz I.14.9 of \cite{hup}).

Finally, suppose that $G$ is the semi-direct product of a nontrivial $2$-group $P$ acting on a group $Q$ with odd order where $P$ has an abelian subgroup $B$ of index $2$ and either (1) $|Z(P)| = 2$ or (2) $|P| \leqslant 4$ and $|Q| > 1$, so that $B$ centralizes $Q$ and any element of $P$ outside of $B$ inverts every element of $Q$.  Let $A = B \times Q$.  Observe that $|G:A| = |PQ:BQ| = |P:P \cap BQ| = |P:B| = 2$.  Since every element of $P$ outside $B$ inverts all every element of $Q$, we conclude that $Z(G) \cap Q = 1$, so $Z(G) \leqslant P$.  This implies that $Z(G) \leqslant Z(P)$.  If $|Z(P)| = 2$, then $|Z(G)| \leqslant 2$.  If $|P|\leqslant 4$, then there exist nontrivial elements of $Q$, so the elements of $P$ outside of $B$ cannot be central in $G$, so $Z(G) \leqslant B$, and since $|B| \leqslant 2$, we conclude that $|Z(G)| \leqslant 2$.
\end{proof}

We now classify the groups $G$ that have a $2$-split decomposition with respect to a normal abelian group $A$.

\begin{lm}\label{lm-e1}
Let $G$ be a group.  Then $G$ has a strict $2$-split decomposition for a normal abelian subgroup $A$ if and only if one of the following occurs:
\begin{enumerate}
         \item $|G/A| = 2$ and $|Z (G)| \leqslant 2$ and $|A| \geqslant 4$.
         \item $|G/A| = 3$ and $G$ is a Frobenius group with Frobenius kernel $A$.
         \item $|G/A| = 4$ and $G$ is either the quaternion group of order $8$ or the dihedral of order $8$ (note that $A$ must be the center of $G$).
         \item $|G/A| = 6$ and $G$ is $S_3$ with $A = 1$ or $S_4$ and $A$ is the normal Klein $4$-subgroup.
\end{enumerate}
\end{lm}
\begin{proof}
Suppose that $G$ has a strict $2$-split decomposition for the normal abelian subgroup $A$.  By Lemma \ref{lm-e0} (4), we know that if $b \in G \setminus A$, then $o (Ab) \leqslant 3$.  This implies that every element of $G/A$ has order at most $3$.

Assume that $G/A$ has an element of order $3$.  Let $T$ be a Sylow $3$-subgroup of $G$.  By Lemma \ref{lm-e0-d}, we see that if $|T| \geqslant 9$, then $T \leqslant A$ which is a contradiction.  Thus, we must have $|T| = 3$.  This implies that $3$ does not divide $|A|$.  If $1 \ne x \in T$, then by Lemma \ref{lm-e0}, we have $|C_A (x)| \leqslant 2$.  Suppose $|C_A (x)| = 2$, and let $a$ be the nonidentity element of $C_A (x)$.  Then $ax$ is an element of $G \setminus A$ and $o (ax) = 6$ which contradicts $o (ax) \leqslant 4$.  Thus, $|C_A (x)| = 1$.  This implies that $AT$ is a Frobenius group when $A > 1$.  If $G = AT$, then we must have $A > 1$ and so, we have conclusion (2).  Note that if $|G/A|$ is odd, then $G/A$ only contains elements of order $3$, so $G = AT$, and this completes the result in this case.

Thus, we may assume that $|G/A|$ is even, and so $G/A$ contains an involution, say $Ai$.  By Lemma \ref{lm-e0} (3), $|C_A (i)| \leqslant 2$.  Thus, $H = A \langle i \rangle$ has a normal subgroup $A$ of index $2$.  If $|A| \geqslant 3$, then $i$ is not in $Z(H)$ which implies that $Z (H) \leqslant A$.  This implies that $Z(H) \leqslant C_A (i)$ and hence $|Z (H)| \leqslant 2$ when $|A| \geqslant 3$.  If $G = H$, then since the split decomposition is strict, we must have $|A| \geqslant 4$ which gives conclusion (1).  Note that if $|G/A| = 2$, then this completes the result in this case.

We now assume that $|G/A|$ is even and at least $4$.  Suppose that $4$ divides $|G/A|$.  Thus, there exist involutions $Ai$ and $Aj$ so that $\langle Ai, Aj \rangle$ is a subgroup of $G/A$ of order $4$.  We may assume that $i$ and $j$ are $2$-elements.  Let $D = \langle i, j \rangle$.  Observe that $DA/A = \langle Ai, Aj \rangle$, so $|D:D \cap A| = |DA/A| = 4$, so by Lemma \ref{lm-e0} (1), we conclude that $D$ is not abelian.  Let $Q$ be the Hall $2$-complement of $A$.  Notice that $D \cap A$ is normalized by both $D$ and $A$, so $D \cap A$ is normal in $DA$.  Notice that $DQ/(D \cap A)$ is the semidirect product of $D/(D \cap A)$ acting on $Q$.  Every element of $D/(D \cap A)$ will act fixed-point freely on $Q$, so if $Q > 1$, then $DQ/(D \cap A)$ is a Frobenius group with Frobenius complement $D/(D\cap A)$.  However, $D/(D \cap A)$ is abelian and not cyclic which is a contradiction to being a Frobenius complement.  This implies that $Q = 1$ and $A$ is a $2$-group.

If $|A| \geqslant 8$, then $A$ has index $2$ in $H = A \langle i \rangle$ and $|Z (H)| \leqslant 2$.  By Lemma \ref{index2}, we know that $H$ is dihedral, semi-dihedral, or generalized quaternion.  In these groups, the only abelian subgroup of index $2$ is cyclic, and so, we deduce that $A$ is cyclic.  Let $C$ be the subgroup of $A$ of order $4$.  Since $A$ is cyclic, $C$ is characteristic which implies that $C$ is invariant under the action of $j$.  We let $C = \langle c \rangle$.  Since $i$ and $j$ do not centralize $c$, we must have $c^i = c^3 = c^j$.  We conclude that $c^{ij^{-1}} = c$, and so, $ij^{-1}$ is element of $G$ outside of $A$ that centralizes $C$.  In particular, $|C_A (ij^{-1})| \geqslant 4$ which contradicts Lemma \ref{lm-e0} (3).  Thus, we must have $|A| \leqslant 4$.

Suppose that $|A| = 4$ which implies that $|DA| = 16$.  Since $DA$ is a $2$-subgroup, we deduce that $|Z(DA)| \geqslant 2$.  Observe that $DA/A$ is abelian, so $(DA)' \leqslant A$.  If $(DA)' = A$, then $|DA:(DA)'| = 4$.  This implies that $DA$ is dihedral, semi-dihedral, or generalized quaternion.  It follows that $DA$ contains a cyclic subgroup $C$ of index $2$ and order $8$ in contradiction to Lemma \ref{lm-e0} (1).  Thus, we must have $|(DA)'| = 2$ since we know that $D$ is nonabelian.  By Lemma \ref{lm-e0} (3), $|C_A (i)| = |C_A (j)| = |C_A (ij)| = 2$.  This implies that $|Z(DA)| = 2$.  We know $DA$ is nonabelian, so $(DA)' > 1$, and thus, $(DA)' \cap Z(DA) > 1$.  We see that $Z(DA) = (DA)'$.  If $g \in DA$, then $[g^2,h] = [g,h]^2 = 1$ for all $h \in DA$.  It follows that $DA/(DA)'$ is elementary abelian, and so, $DA$ is extra-special.  This contradicts $|DA:Z(DA)| = 8$.  Therefore, $|A| = 2$ which implies that $DA$ has order $8$, and thus, $DA$ is either the dihedral group or the quaternion group.  If $|G/A| = 4$, then $G =DA$ and we have conclusion (3).

Continuing to assume that there exists involutions $Ai$ and $Aj$ as above, suppose that $3$ divides $|G/A|$.  Hence, there exists an element $k$ of order $3$ in $G \setminus A$.  It follows that $A \langle k \rangle$ is a cyclic subgroup of order $6$ that is not contained in $A$ which contradicts Lemma \ref{lm-e0} (1).  Thus, $3$ does not divide $|G/A|$.  On the other hand, suppose there is an involution $Al$ that is not contained in $A \langle i, j \rangle$.  Now, $E = A \langle i, j, l \rangle$ is a group of order $16$.  Notice that $Z(E) = A$ and $E/A$ is abelian, so $A = E'$.  As in the previous paragraph, this implies that $E$ is an extra-special group which is a contradiction since $|E:A| = |E:E'| = 8$.  We conclude that if $4$ divides $|G:A|$ then $|G:A| = 4$.

Finally, we know that the odd part of $|G/A|$ is at most $3$.  It follows that the remaining possibility is that $|G/A| = 6$.  Let $Q$ be the Hall $2$-complement of $A$.  Let $P$ be the Sylow $2$-subgroup of $A$ and let $T$ be a Sylow $2$-subgroup.  Observe that $P$ is an abelian subgroup of index $2$ in $T$.  Let $i$ be an element of $T \setminus P$, and observe that $i$ lies in $G \setminus A$.  Applying Lemma \ref{lm-e0} (3), $|C_A (i)| \leqslant 2$ which implies that $|C_P (i)| \leqslant 2$ and so, $|Z (T)| \leqslant 2$.  In light of Lemma \ref{index2}, we see that $T$ is dihedral, semi-dihedral, or generalized quaternion.  Since the only abelian subgroup of index $2$ in those groups are cyclic, we conclude that $P$ is cyclic.  We know that a cyclic $2$-group of order at least $8$ has no odd order automorphisms.  We conclude that the elements of order $3$ in $G \setminus A$ centralize $P$ which contradicts Lemma \ref{lm-e0} (3).  We determine that $|P| \leqslant 4$.  If $|P| = 2$, then the elements of order $3$ in $G \setminus A$ would centralize $P$.  This would imply that $G \setminus A$ would contain an element of order $6$ which contradicts Lemma \ref{lm-e0} (3).  Thus, we have that either $|P| = 1$ or $|P| = 4$.  Notice that if $|P| = 4$, then $P$ is a Klein $4$-group and $T$ is the dihedral group of order $8$ in this case.  Let $R$ be a Sylow $3$-subgroup of $G$.  Since $G/A$ has order $6$, we see that $AR/A$ is normal in $G/A$.   This implies that $PR$ is normal in $G$, and so, $RT = R(PT)$ is a subgroup of $G$.  If $P = 1$, then since $RT$ has order $6$ and intersects $A$ trivially, we cannot have $RT$ abelian.  On the other hand, we see that if $|P| = 4$, then $C_T (P) = P$, and so, $C_{RT} (P)$ which implies that $RT/P$ is isomorphic to a subgroup of the automorphism group of $P$ which is $S_3$.  It follows that $RT/P$ is isomorphic to $S_3$.  Let $Q$ be the Hall $2$-complement of $A$.  If $Q > 1$, note that Lemma \ref{lm-e0} (3) implies that $RT/P$ acts Frobeniusly on $Q$ which is a contradiction since $S_3$ cannot be a Frobenius complement.  Therefore, we conclude that $Q = 1$, and $G$ is either $S_3$ with $A = 1$ or $S_4$ with $A$ being the normal Klein $4$-subgroup.

Conversely, if $G$ has a normal subgroup $A$ of index $2$ with $|Z(G)| \leqslant 2$, then Lemma \ref{exist} implies that $G$ has a $2$-split decomposition.  If $G$ is a Frobenius group with abelian Frobenius complement $A$ of index $3$, then Lemma \ref{Frobenius} may be used to show that $G$ has a $2$-split decomposition with respect to $A$.  We have provided above, $2$-split decompositions for $G$ with respect to $A$ when $G$ is either the dihedral group or the quaternion group of order $8$ and $A$ is the center of $G$, and when $G$ is $S_3$ and $A= 1$.  Finally, when $G$ is $S_4$ and $A$ is the normal Klein $4$-subgroup, we have the $2$-split decomposition:
$$
\begin{array}{lcl}
A & = & \{ (1), (12)(34),(13)(24),(14)(23) \} \\[0.1cm]
B_1 & = & \{ (123),(124),(134),(234),(1234),(1243),(1324),(12),(13),(14) \} \\[0.1cm]
B_2 & = & \{ (132),(142),(143),(243),(1432),(1342),(1423),(34),(24),(23)\}.\\[0.1cm]
\end{array}$$
The lemma is proved. \end{proof}

Next, we classify the groups $G$ that have a strict $2$-split decomposition with respect to a nonnormal abelian subgroup $A$.

\begin{lm} \label{2-split-nonn}
Let $G$ be a nonabelian group.  Then $G$ has a strict $2$-split decomposition with respect to a nonnormal abelian group $A$ if and only if one the following occurs:
\begin{enumerate}
\item $G \cong S_3$ with $|G:A| = 3$.  I.e., $A$ is a nonnormal subgroup of order $2$.
\item $G \cong A_4$ with $|G:A| = 6$.  I.e. $A$ is a nonnormal subgroup of order $2$.
\end{enumerate}
\end{lm}

\begin{proof}
First of all, for every abelian subgroup $U$ of $G$, either $U \leqslant A$ or $|U| \leqslant 4$ by Lemma \ref{lm-e0} (1).  Since $A$ is a {\em nonnormal} subgroup of $G$, we have $A^g \nleqslant A$, for some $g\in G$, which is an abelian subgroup of $G$.  Since $A^g$ is not contained in $A$, we obtain $|A| = |A^g| \leqslant 4$. We conclude that all abelian subgroups of $G$ have order at most $4$.  By Lemma \ref{lm-e0} (3), the elements outside $A$ have orders at most $4$.   We deduce that $G$ is a $\{2, 3\}$-group.  Since every group of order $16$ contains an abelian subgroup of order $8$ and every group of order 9 is abelian, this forces $|G|=2^a\cdot 3^b$ where $a \leqslant 3$ and $b\leqslant 1$.  Now, it is easy to check that $G$ is isomorphic to one of the following groups: $S_3$, $D_8$, $Q_8$, $A_4$, or  $S_4$.  In $S_3$, if $A$ is nonnormal, then $|G:A| = 3$.  For $A_4$, notice that the nonnormal subgroups have order $2$ or $3$.  Notice that for subgroups of order $3$, the existence of the Klein $4$-subgroup and Lemma \ref{lm-e0} (1) would imply that $n \geqslant 3$ which is a contradiction.  Thus, we must have $|A| = 2$ which yields $|G:A| = 6$.   For $S_4$, we can find an abelian subgroup $U$ of order $4$ that intersects $A$ trivially.  To see this, observe that if $|A| = 3$, then we can take $U$ to be any abelian subgroup of order $4$.  If $A$ intersects the Klein $4$-subgroup trivially, take $U$ to be the Klein $4$-subgroup.  Otherwise, $A$ will intersect the Klein $4$-subgroup in a subgroup of order $2$.  Take $U$ to be a cyclic subgroup of order $4$ that intersects the Klein-subgroup in a different subgroup of order $2$, and it follows that $U$ and $A$ will intersect trivially.  This implies by Lemma \ref{lm-e0} (1) that any $n$-split decomposition for $A$ must have $n \geqslant 3$.  To see that $D_8$ and $Q_8$ cannot have an $n$-split decomposition with $A$ nonnormal, observe that since the decomposition is strict, we have $Z (G) \leqslant A$, and this implies $A$ is normal when $G$ is $D_8$ or $Q_8$.

Conversely, suppose that $G$ is one of the groups given and $A$ has the given index.  We now show that these groups have a $2$-split decomposition for the given one of the possible $A$'s.  In each case, we can obtain $2$-decompositions for the other possible $A$'s by conjugating.   When $G = S_3$, we take $$A = \{1, (12)\},  \ \ B_1=\{(13), (123)\}, \ \  B_2=\{(23), (132)\}.$$  For $G = A_4$ and $|A| = 2$, we present the $2$-split decomposition:
$$\begin{array}{lcl}
A & = & \{ (1), (12)(34) \}; \\[0.1cm]
B_1 & = & \{ (13)(24), (123),(124),(134),(234) \}, \\[0.1cm]
B_2 & = & \{ (14)(23), (132),(142),(143),(243) \}.\\[0.1cm]
\end{array}$$
The lemma is proved.
\end{proof}

\section{Groups having a strict $3$-split decomposition}
We now work to determine which groups have a strict $3$-split decomposition.

\begin{lm} \label{2-3-split}
If $G$ has a strict $2$-split decomposition over $A$, then $G$ has a strict $3$-split decomposition over $A$ except for the following: (1) $G \cong S_3$ (with $|A| = 2$), (2) $G \cong D_8$ or $Q_8$ and $|A| = 4$, or (3) $G \cong D_{10}$ (with $|A| = 5$).
\end{lm}
\begin{proof}
Let $G$ be a group having a strict $2$-split decomposition over $A$;  say $G = A \uplus B_1 \uplus B_2$.  If either of $B_1$ or $B_2$ contains at least four elements, then it follows from Lemma \ref{lm-ee0} (a) that $G$ has a strict $3$-split decomposition, and the result is proved.  We assume, therefore, that $|B_1| \leqslant 3$ and $|B_2| \leqslant 3$.  Since $|A|$ divides $|B_1 \uplus B_2|$, we have $|A| \leqslant 6$ and so $|G| \leqslant 12$.  It is now easy to check that $G$ has a strict  $3$-split decomposition, except the cases mentioned above.
\end{proof}

Please note that Lemma \ref{2-3-split} does not imply that the groups $G$ mentioned in the conclusion do not have strict $3$-split decompositions, just that the strict $2$-split decompositions for the given $A$'s do not yield strict $3$-split decompositions.  Indeed, the strict $2$-split decompositions for $D_8$ and $Q_8$ over their centers do yield strict $3$-split decompositions.  On the other hand, since $S_3$ only has $6$ elements, it is not possible for $S_3$ to have a strict $3$-split decomposition.
Let $G=D_{10}$ and let $A$ be an abelian subgroup of $G$ that does not  have order $5$.
Then $A$ will intersect the subgroup of order $5$ trivially, and if $G$ has a strict $n$-split decomposition with respect to $A$, then $n\geqslant 4$ by Lemma \ref{lm-e0} (1).

\begin{lm} \label{3-split prel}
Suppose $G$ has a strict $3$-split decomposition of $G$ with respect to a normal abelian subgroup $A$.  Then the following are true:
\begin{enumerate}
\item Every element of $G/A$ has order at most $4$.
\item $G/A$ is a $\{ 2, 3 \}$-group.
\item If $3$ divides $|G/A|$, then a Sylow $3$-subgroup of $G$ has order $3$.
\end{enumerate}
\end{lm}

\begin{proof}
Note that $G/A$ does not contain any element of order greater than or equal to $5$ by Lemma \ref{lm-e0-d}.  Also, by Lemma \ref{lm-e0-d}, we know that if $3$ divides $|G/A|$, then a Sylow $3$-subgroup of $G$ has order $3$.
\end{proof}

\begin{lm} \label{3-split-D8}
Suppose the nonabelian group $G$ has a strict $3$-split decomposition of $G$ with respect to a normal abelian subgroup $A$ of odd order.
\begin{enumerate}
\item If a Sylow $2$-subgroup $S$ of $G$ is isomorphic to $D_8$, then one of the following occurs:
     \begin{enumerate}
     \item $A = 1$ and $G \cong S_4$
     \item $G = AS$, such that $S = \langle x, y \mid x^4 = y^2 = x^y x = 1 \rangle ( = D_8)$ and $A = \langle u, v \mid u^3 = v^3 = [u,v] = 1 \rangle,$ where $u, v$ can be chosen such that $u^x = v$, $v^x = u^{-1}$, $u^y = u$, $v^y = v^{-1}$.
     \end{enumerate}
\item If a Sylow $2$-subgroup $S$ of $G$ is isomorphic to $Q_8$, then $G = AQ_8$ is a Frobenius group.
\end{enumerate}

Furthermore, each of the three groups listed have a strict $3$-split decomposition with respect to $A$.
\end{lm}

\begin{proof}
Suppose $G/A$ is not a $2$-group and has a Sylow $2$-subgroup that is nonabelian of order $8$.  It follows that $G/A$ has order $24$ by Lemma \ref{3-split prel}.  If $G/A$ has a normal Sylow $2$-subgroup, then an element of order $3$ will centralize the center of $S$ and this yields an element of order $6$ which is a contradiction.  Thus, a Sylow $2$-subgroup of $G/A$ is not normal, and hence $G/A$ has three Sylow $2$-subgroups; so the action of $G/A$ on its Sylow $2$-subgroups yields a homomorphism into $S_3$.  It is not difficult to see that this implies that $G/A$ is isomorphic to $S_4$ when $G/A$ is not a $2$-group.

We may assume that $S$ is either $D_8$ or $Q_8$.  Let $t$ be the central involution of $S$.  By Lemma \ref{lm-e0} (3), we have $|C_A (t)| \leqslant 3$.  If $C_A (t)\neq 1$, then $|C_A (t)| = 3$.  Observe that $S$ acts on $C_A (t)$, and if $K = C_S (C_A (t))$, then $|S/K| = 2$, so $|K| = 4$.  Notice that $C_A (t) K$ is an abelian subgroup of $G$ of order $12$ that is not contained in $A$ and this contradicts Lemma \ref{lm-e0}(1) which shows that such a subgroup has size at most $2 \cdot 3 = 6$.  Thus, $C_A (t) = 1$ and so, $t$ inverts every element of $A$.

Now assume that $S$ is isomorphic to $D_8=\langle x, y | x^4=y^2=x^yx=1\rangle$ as above, and this implies that $G/A$ is either $D_8$ or $S_4$.  Suppose $A = 1$, then since $G$ has a strict $3$-split decomposition, we cannot have $G \cong D_8$ by Lemma \ref{lm-e0} (5).  Suppose that $A > 1$. The previous paragraph implies that $|C_A (x^2)| = 1$ and hence $|C_A (x)| = 1$.  Since $S = D_8$ is not a Frobenius complement, we must have $C_A(y) \neq 1$ and $|C_A(y)|=3$ by Lemma \ref{lm-e0} (1).  Recall that Lemma \ref{lm-e0-d} implies that $3$ does not divide $|G/A|$ since $3$ divides $|A|$. We conclude that $G/A$ is a $2$-group and $G = A D_8$.  By Fitting's lemma, we have $A = C_A(y) \times [A,y]$ and $y$ inverts every element of $[A,y]$.  Clearly, $[A,y]\ne 1$ and $yx^2$ centralizes $[A,y]$. Thus, $[A,y] \leqslant C_A (yx^2)$ which implies that $|[A,y]| = 3$ and $|A| = 9$.  Let $1 \ne u \in C_A (y)$ and $v=u^x$. Then $A = \langle u, v \rangle$, and we have  $$v^y= u^{xy}=u^{yx^{-1}}=u^{x^{-1}}=(u^x)^{x^2}=v^{x^2}=v^{-1},$$ and $v^x=u^{x^2}=u^{-1}$.

Assume $S$ is isomorphic to $Q_8$.  Since a Sylow $2$-subgroup of $S_4$ is dihedral, we do not have $G/A \cong S_4$, and so, $G/A \cong Q_8$ which implies that $G = AS$.  Lemma \ref{lm-e0} (5) implies that $A > 1$.  Let $t$ be the involution in $S$. From the second paragraph, we have $|C_A(t)|=1$, and so, $\langle t\rangle$ acts fixed-point-freely on $A$.  This implies that $Q_8$ acts fixed-point-freely on $A$; hence $G = AS$ is a Frobenius group with abelian kernel $A$.

Conversely, we now show that the three groups named have strict $3$-split decompositions for the given $A$.

For $G = S_4$ and $A = 1$, we obtain a $3$-split decomposition as follows:
$$\begin{array}{ccl}
B_1 & = & \{ (12)(34), (13), (14), (1234), (1243), (123), (142), (234) \}, \\[0.1cm]
B_2 & = & \{ (13)(24), (12), (23), (1324), (1342), (132), (134), (243) \}, \\[0.1cm]
B_3 & = & \{ (14)(23), (34), (24), (1423), (1432), (124), (143) \}.\\[0.1cm]
\end{array}$$

Suppose $G = AS$, such that $S = \langle x, y \mid x^4 = y^2 = x^y x = 1 \rangle ( = D_8)$ and $A = \langle u, v \mid u^3 = v^3 = [u,v] = 1 \rangle,$ where $u, v$ can be chosen such that $u^x = v$, $v^x = u^{-1}$, $u^y = u$, $v^y = v^{-1}$.  Observe that $C_A (y) = \langle u \rangle$, $C_A (x^2y) = \langle v \rangle$, $C_A (xy) = \langle uv^2 \rangle$, and $C_A (x^3y) = \langle uv \rangle$.  We have $$Ay = [A,y]y \cup [A,y]uy \cup [A,y] u^2 y.$$  Similarly, $$Ax^2y = [A,x^2y] x^2y \cup [A,x^2y] v x^2y \cup [A,x^2y] v^2x^2y,$$ $$Axy = [A,xy] xy \cup [A,xy] uv^2 xy \cup [A,xy] u^2v xy, \ \   $$ and $$ \ \ \ \ \ Ax^3y = [A,x^3y] x^3y \cup [A,x^3y] uv x^3y \cup [A,x^3y] u^2v^2 x^3y.$$

Note first that $x^2$ inverts both $u$ and $v$, so $x^2$ inverts every element of $A$.  This implies that $A \langle x \rangle$ is a Frobenius group.  Thus, no pair of elements in any of the cosets $Ax$, $Ax^2$, and $Ax^3$ will commute with each other, and $Ax$ is contained in the conjugacy class of $x$, $Ax^2$ is contained in the conjugacy class of $x^2$, and $A x^3$ is contained in the conjugacy class of $x^3$.  Note that $x^y = x^3$; so $Ax \cup Ax^3$ is contained in the conjugacy class of $x$.  Observe that $|{\rm cl} (x)| \geqslant 2|A| = 18$.  On the other hand, $\langle x \rangle \leqslant C_G (x)$, and thus, $4 \leqslant |C_G (x)| = |G|/|{\rm cl} (x)| \leqslant 72/18 = 4$.  We conclude that $C_G (x) = \langle x \rangle$ and ${\rm cl} (x) = Ax \cup Ax^3$.  Also, we know that $|{\rm cl} (x^2)| \geqslant |A| = 9$ and $S \leqslant C_G (x)$, so $8 = |S| \leqslant |C_G (x^2)|=|G|/|{\rm cl} (x^2)| \leqslant 72/9 = 8$.  We deduce that $C_G (x^2) = S$ and ${\rm cl} (x^2) = Ax^2$.

We see that $C_G (x)$ contains one element from each of $Ax$, $Ax^2$ and $Ax^3$.  In particular, $C_G (x)$ does not intersect $$[A, y]y, \ [A, xy]xy,  \ [A, x^2y]v x^2y, \ {\rm or} \  [A, x^3y]uvx^3y.$$  For any element $g$ of $Ax$, the centralizer of $g$ in $G$ will be conjugate to $C_G (x)$, and thus, we deduce that $g$ commutes with no element in these four cosets.  It follows that no element of $Ax$ commutes with any element in these four cosets.  Similarly, no element of $Ax^3$ commutes with any element in $[A,y]uy$, $[A,xy]u^2vxy$, $[A,x^2y]x^2y$, or $[A,x^3y]x^3y$.

Note that $[A,y] \langle y \rangle$ will be a Frobenius group, so $y$ is conjugate to all the elements in $[A,y] y$ and no two elements in $[A,y]y$ commute with each other.  Similarly, $xy$, $x^2y$, and $x^3y$ (respectively) are conjugate to all of the elements in the cosets $[A,xy] xy$, $[A,x^2y] x^2y$, and $[A,x^3y] x^3y$ (respectively) and no two elements in any of those cosets will commute.  We see that $y^x = x^2 y$, so $[A,y]y \cup [A,x^2y] x^2y \subseteq {\rm cl} (y)$.  This implies that $|[A,y]| + |[A,x^2y]| = 3 + 3 = 6 \leqslant |{\rm cl} (y)|$.  On the other hand, we know that $\langle u \rangle \langle y, x^2 \rangle \leqslant C_G (y)$.  This implies that $12 \leqslant |C_G (y)| = |G|/|{\rm cl} (y)| \leqslant 72/6 = 12$.  We determine that ${\rm cl} (y) = {\rm cl} (x^2y) = [A,y]y \cup [A,x^2y] x^2y$ and $C_G (y) = \langle u \rangle \langle y, x^2 \rangle$.  In a similar fashion, one can see that ${\rm cl} (xy) = {\rm cl} (x^3y) = [A,xy] xy \cup [A,x^3y] x^3y$ and $C_g (xy) = \langle uv^2 \rangle \langle xy, x^2 \rangle$.

We have now that $C_G (x^2)$ consists of elements in $Ax$ $Ax^2$, $Ax^3$, $[A,y]y$, $[A,xy]xy$, $[A,x^2y]x^2y$, $[A,x^3y]$ and does not intersect $[A,y]u^2y$, $[A,xy]u^2v xy$, $[A,x^2y]v^2 x^2y$, or $[A,x^3y]u^2v^2 x^3y$.  Using conjugacy, we see that no element in $Ax^2$ will commute with any element in $[A,y]u^2y$, $[A,xy]u^2v xy$, $[A,x^2y]v^2 x^2y$, or $[A,x^3y]u^2v^2 x^3y$.

Since $u$ commutes with $y$ and the elements in $[A,y]$, we see that all the elements in $[A,y]uy$ are conjugate as are the elements in $[A,y]u^2y$ and no two elements in either of those two cosets will commute.  Observe that $$(uy)^{x^2} = u^2 y, \ \  (uy)^x = v x^2y \ \  {\rm and} \ \  (uy)^{x^3} = v^2 x^2y.$$  Arguing as above, we can show that $${\rm cl} (uy)=[A,y]uy \cup [A,y]u^2y \cup [A,x^2y] vx^2y \cup [A,x^2y] v^2 x^2y,$$ and $C_G (uy) = \langle uy \rangle$ is cyclic of order $6$.  Similarly, we can obtain $${\rm cl} (u^2vxy) = [A,xy]u^2v xy \cup [A,xy] uv^2 xy \cup [A,x^3y]uv x^3y \cup [A,x^3y] u^2v^2 x^3 y,$$ and $C_G (u^2vxy) = \langle u^2v xy \rangle$ is cyclic of order $6$.

We have that $C_G (y) = \{ 1, u, u^2, y, x^2, x^2y, uy, ux^2, ux^2y, u^2y, u^2x^2, u^2x^2y \}$.  Notice that $u^{x^2y} = u^{-1}$, so $[A,x^2y] = \{ 1, u, u^2 \}$.  It follows that $C_G (y)$ consists of three elements of $A$, three elements of $Ax^2$, the coset $[A,x^2y] x^2y$ and one element in each of $[A,y]y$, $[A,y]uy$, and $[A,y]u^2y$. So $y$ does not commute with any elements of the cosets $[A,x^3y]x^3y$, $[A,x^2y]vx^2y$, and $[A,x^3y]uvx^3y$.  Noting that conjugating, we see that this applies to all the elements in the coset $[A,y]y$.

Observe that $C_G (uy) = \{ 1, uy, u^2, y, u, u^2 y \}$, and so, $C_G (uy)$ contains three elements of $A$, and one element in each of $[A,y]y$, $[A,y]uy$, and $[A,y]u^2y$.  Working the same way, we can see that $C_G (vuy)$ and $C_G (v^2uy)$ are composed from the same number of elements in the same sets.  Thus, we see that no element in $[A,y]uy$ commutes with any element in the cosets $[A,x^2y] v^2 x^2y$, $[A,xy]u^2vxy$, and $[A,x^3y] u^2v^2 x^3y$.

Using similar arguments, we can show that the following sets form a strict $3$-split decomposition of $G$ with respect to $A$:
$$\begin{array} {ccl}
B_1 & = & Ax \cup [A,y]y \cup [A,xy]xy \cup [A,x^2y]v x^2 y \cup [A,x^3y]uv x^3y, \\[0.2cm]
B_2 & = & Ax^2 \cup [A,y]uy \cup [A,xy]u^2v xy \cup [A,x^2y]v^2 x^2y \cup [A,x^3y] u^2v^2x^3y, \\[0.2cm]
B_3 & = & Ax^3 \cup [A,y]u^2y \cup [A,xy]uv^2 xy \cup [A,x^2y] x^2y \cup [A,x^3y]x^3y.\\[0.1cm]
\end{array}$$

Finally, suppose that $G = AS$ is a Frobenius group with abelian Frobenius kernel $A$ and Frobenius complement $Q_8 = \{ \pm 1, \pm i, \pm j, \pm k \}$.  We claim to obtain a strict $3$-split decomposition for $G$ with respect to $A$ by taking $$B_1 = A(-1),  \ \  B_2 = Ai \cup Aj \cup Ak \ \ {\rm and} \ \  B_3 = A(-i) \cup A(-j) \cup A(-k).$$  To see this, observe that since $G$ is a Frobenius group, we know that no two elements in any coset of $A$ that is not $A$ will commute.   Suppose $g \in Ai$ and $h \in Aj$.  When $g$ and $h$ lie in different Frobenius complements, then they do not commute.  However, if they lie in the same Frobenius complement, they will be two elements of order $4$ that are not inverses.  In $Q_8$, this implies that they do not commute.  Similarly, we can show that no two elements of either $B_2$ or $B_3$ commute proving our claim.
\end{proof}

\begin{lm} \label{3-split-odd}
Let $G$ be a nonabelian group, and let $S$ a Sylow $2$-subgroup of $G$.  Then $G$ has a strict $3$-split decomposition with respect to a normal abelian subgroup $A$ of odd order if and only if one of the following holds:
\begin{enumerate}
\item $|S| = 1$ and $G$ is a Frobenius group, with Frobenius kernel $A$ satisfying $|G:A| = 3$.
\item $|S| = 2$ and $G$ is a Frobenius group with Frobenius kernel $A$ satisfying $|A| \geqslant 7$ and $|G:A| = 2$.
\item $|S| = 2$ and $G = Z_3 \times F$ where $F$ is a Frobenius group with Frobenius kernel $A \cap F$ and Frobenius complement $S$.
\item $S$ is cyclic of order $4$ and $G$ is a Frobenius group with kernel $A$ satisfying $|G:A| = 4$.
\item $S$ is cyclic of order $4$ and $G = AS$ such that $A = [A,S]$ and $A = C_A(t) \times [A, t]$, where $t$ is the involution in $S$, $|C_A (t)| = 3$, and $[A,t] > 1$.
\item $S$ is a Klein $4$-group, $A = 1$, and $G \cong A_4$.
\item $S$ is a Klein $4$-group and $G = AS$, where $A$ is elementary abelian of order $9$ or $27$, and either $G = S_3 \times S_3$ where $A = A_3 \times A_3$ or $G= AS$ satisfies $A= C_A (s_1) \times C_A (s_2) \times C_A (s_3)$ where $S=\{1, s_1, s_2, s_3\}$ and $|C_A (s_i)| = 3$ for $i = 1, 2, 3$.
\item $S \cong D_8$ and $G$ satisfies one of the groups in Conclusion (1) of Lemma \ref{3-split-D8}.
\item $S \cong Q_8$ and $G = AQ_8$ is a Frobenius group.
\end{enumerate}
\end{lm}

\begin{proof}
We suppose first that $G$ has a strict $3$-split decomposition with respect to $A$.  If $|S| = 1$, then $|G|$ is odd.  We deduce that $G/A$ has order $3$ and $3$ does not divide $|A|$ by Lemma \ref{3-split prel}.  Let $b \neq 1$ be a $3$-element in $G\backslash A$.  Then $|C_A(b)| \leqslant 3$ by Lemma \ref{lm-e0} (3).  Since $|A|$ is odd and not divisible by $3$, we conclude that $C_A (b) = 1$. Thus $G = A \langle b \rangle$ is a Frobenius group.

Let $|S| = 2$ and consider the involution $t$ in $S$. By Lemma \ref{lm-e0} (3), we know that $|C_A (t)| \leqslant 3$.  Since $|A|$ is odd, either $|C_A (t)| = 1$ or $|C_A (t)| = 3$.  If $|C_A (t)| = 3$, then by Lemma \ref{3-split prel}, we see that $G/A$ is a $2$-group.  Thus, $G = AS$ and by Fitting's lemma, we have $A = C_A (t) \times [A,t]$.  It is not difficult to see that $G = C_A (t) \times [A,t] S$ where $C_A (t) \cong Z_3$ and $F = [A,t]S$ is a Frobenius group with Frobenius kernel $[A,t] = F \cap A$ and Frobenius complement $S$.  We now consider the case $|C_A (t)| = 1$.  If $G = AS$, then it is not difficult to see that $G$ is a Frobenius group with Frobenius kernel $A$ and Frobenius complement $S$.  If $G \ne AS$, then $G/A$ is a $\{ 2, 3 \}$-group and not a $2$-group.  Thus, $3$ divides $|G:A|$ which implies that $|G:A| = 6$ and $3$ does not divide $|A|$.  We can find an element $b \in G$ of order $3$ so that $B = \langle b, t \rangle$ is a Hall $\{ 2,3 \}$-subgroup of $G$.  Notice that $B \cong G/A$, so $B$ is not cyclic.  On the other hand, arguing as in the first paragraph of this proof, we have $C_A (b) = 1$ which would imply that $B$ is a Frobenius complement and this is a contradiction since $B$ not abelian implies that $B \cong S_3$ and $S_3$ cannot be a Frobenius complement.  This completes the case when $|S| = 2$.

Suppose $S$ is cyclic of order $4$. Let $t$ be the involution in $S$.  As in the case $|S| = 2$, one can prove that $|C_A(t)|$ equals either $1$ or $3$.  If $|C_A (t)| = 3$, then $3$ does not divide $|G:A|$.  This implies that $G = AS$.  By Lemma \ref{lm-e0} (2), we have $(|S|-1)|C_A (S)| \leqslant 3$ which implies that $|C_A (S)| = 1$ and hence, $A = [A,S]$.  By Fitting's lemma, we have $A = [A,t] \times C_A (t)$.  Notice that if $[A,t] = 1$, then $t$ will be central in $G$, and this violates Lemma \ref{lm-e0} (5), so we must have $[A,t] > 1$.

We now assume that $C_A (t) = 1$.  This implies that $AS$ is a Frobenius group.  If $3$ divides $|G:A|$, then we can choose an element $b$ of order $3$ so that $\langle S, b \rangle$ is a Hall $\{ 2, 3 \}$-subgroup of $G$.  Notice that $b$ and $t$ will have to commute which gives $G/A$ an element of order $6$ which is a contradiction.  Thus, $G = AS$ as desired.

Let $S$ be an elementary abelian subgroup of order $4$. Since $S$ cannot be a Frobenius complement, either $|A| = 1$ or there exists an involution $t$ in $S$ such that $|C_A(t)| = 3$.  Suppose that $|A| = 1$.  Since $G$ is not abelian, we must have $|G| = 12$, and since $A_4$ is the only nonabelian group of order $12$ that does not have an element of order $6$, we deduce that $G \cong A_4$.

Suppose now that there exists an involution $t$ so that $|C_A (t)| = 3$.  By Lemma \ref{lm-e0-d}, $A$ contains a Sylow $3$-subgroup of $G$ and hence $G = AS$. Let $r$ be an involution in $S$ such that $r \neq t$.  If $C_A (r) \leqslant C_A (t)$ then $C_{[A, t]} (r) = 1$.  It follows that $[A,t] \leqslant [A,r]$.  Since both $r$ and $t$ invert all the elements of $[A,t]$, it follows that $rt$ centralizes $[A,t]$.  In particular, we have $[A,t] \leqslant C_A (rt)$.  Applying Lemma \ref{lm-e0} (3) to $rt$, we obtain $|C_A (rt)| \leqslant 3$ so $|[A, t]| \leqslant 3$ and $|A| \leqslant 3^2$.  If $|A| = 3$, then since $|C_A (t)| = 3$, we have $A = C_A (t)$.  It follows that $t$ is central in $G$, and since the decomposition is strict, this contradicts Lemma \ref{lm-e0} (5).  If $|A| = 3^2$, then $C_A (rt) = [A,t]$.  Since $C_A (r) \leqslant C_A (t)$, we see that every element of $C_A (r)$ is centralized by both $r$ and $t$ and so is centralized by $rt$.  Since $C_A (t) \cap C_A (rt) = 1$, we deduce that $C_A (r) = 1$.  This implies that $r$ inverts every element of $A$.  It follows that $rt$ inverts every element of $C_A (t)$, and so, $G = [A,t] \langle t \rangle \times C_A (t) \langle rt \rangle \cong S_3 \times S_3$.

We may assume that $C_A (t)$, $C_A (r)$, and $C_A (rt)$ are three distinct subgroups of order $3$.  Applying Fitting's lemma, we have $A = C_A (t) \times [A,t]$.  Since $r$ and $t$ commute, it follows that $C_A (r)$ is $t$-invariant.  By Fitting's lemma, we have $C_A (r) = [C_A (r),t] \times C_{C_A (r)} (t)$.  Since $|C_A (r)| = 3$ and $C_A (r) \ne C_A (t)$, we see that $C_A (r) = [C_A (r), t] \leqslant [A,t]$.   Using Fitting's lemma once more, we have $[A,t] = C_A (r) \times [[A,t],r]$.  Again, both $r$ and $t$ will invert all the elements of $[[A,t],r]$ so $rt$ centralizes every element in $[[A,t],r]$.  Notice that $rt$ will invert every element in $C_A (t) C_A (r) = C_A (t) \times C_A (r)$.  Since $C_A (rt)$ is nontrivial, we must have $[[A,t],r] = C_A (rt)$.  We conclude that $A=C_A(t)\times C_A(r)\times C_A(rt)$ and $|A| = 27$.

If $|S| = 8$, then since $S$ is nonabelian, $S\cong Q_8$ or $S\cong D_8$ and the conclusion
follows from Lemma \ref{3-split-D8}.

Finally, if $|S| > 8$, then $S$ contains an abelian subgroup of order $8$ which is impossible by Lemma \ref{lm-e0}.

Conversely, we show that if $G$ is one of the groups mentioned, then $G$ has a strict $3$-split decomposition.  If $G$ is a Frobenius group with abelian Frobenius kernel of odd order and abelian Frobenius complement of order $2$, $3$, or $4$, then the result follows from Lemma \ref{Frobenius} and Lemma \ref{2-3-split}.

Suppose that $G = Z \times F$ where $Z \cong Z_3$ and $F$ is a Frobenius group with abelian Frobenius kernel $B$ and Frobenius complement $S$ with $|S| = 2$.  We write $Z$ for $Z \times 1$ and $F$ for $1 \times F$.  Notice that $Z(G)=Z\cong Z_3$.  Take $A=Z B$.  Let $t$ be the involution in $S$ and let $z$ be a generator of $Z$.  Since $F$ is a Frobenius group, we see that the elements of $Bt$ do not commute.  Since $z$ and $z^2$ are central in $G$, it will follow that the elements of $Bzt$ and $Bz^2t$ do not commute.  It is not difficult to see that $G = A \uplus Bt \uplus Bzt \uplus Bz^2t$ is a strict $3$-split decomposition for $G$.

Next, suppose that $S$ is cyclic of order $4$ and $G = AS$ such that $A = [A,S]$ and $A = C_A(t) \times [A,t]$, where $t$ is the involution in $S$, $|C_A (t)| = 3$, and $[A,t] > 1$.  Let $S = \langle s \rangle$ and $C_A (t) = \langle z \rangle$, and note that $t = s^2$.   Notice that $C_A (s) = 1$, so $C_G (s) = S$.  By Lemma \ref{lm-eee0}, no two elements in $As$ commute, and in a similar fashion, no two elements in $As^3$ commute.  Applying Lemma \ref{lm-eee0} to $t = s^2$, we see that no two elements of $[A,t] t$ commute.  Since $z$ and $z^2$ centralize $t$ and $[A,t]$, we conclude that no two elements in $[A,t] zt$ and no two elements in $[A,t] z^2t$ commute.  Since $A$ is partitioned by $[A,t] \cup [A,t] z \cup [A,t] z^2$, it follows that the coset $At$ is partitioned by $[A,t] t \cup [A,t] zt \cup [A,t] z^2 t$.

The conjugacy class of $s$ in $G$ has size $|G:C_G (s)| = |G:S| = |A|$.  On the other hand, it is easy to see that $G' \leqslant A$ and ${\rm cl} (s) \subseteq G's \subseteq As$.  We see that $|A| = |{\rm cl} (s)| \leqslant |G'| \leqslant |A|$.  This implies that $A = G'$ and ${\rm cl} (s) = As$.  Similarly, we obtain ${\rm cl} (s^3) = As^3$.  We have that $C_G (t) = C_A (t) \times S$.  This implies that $|{\rm cl} (t)| = |[A,t]|$.  Since $[A,t] \langle t \rangle$ is a Frobenius group, we have $[A,t] t \subseteq {\rm cl} (t)$, and hence, we conclude that ${\rm cl} (t) = [A,t] t$.  The centralizer of $s$ has the form $\{1, s, s^2, s^3 \}$ and every element in $As$ is conjugate to $s$.  For any element in $As$, its centralizer will consist of $1$, itself, a conjugate of $s^2$, and a conjugate of $s^3$.  Since the conjugacy classes of $s$, $s^2$, and $s^3$ are $As$, $[A,t] t$, and $As^3$ respectively, we conclude that no element in $As$ will commute with any element of $[A,t] zt$.  Similarly, no element in $As^3$ will commute with an element of $[A,t] z^2t$.  Thus, if we take $B_1 = As \cup [A,t] zt$, $B_2 = [A,t] t$, and $B_3 = As^3 \cup [A,t] z^2t$, then we obtain a strict $3$-split decomposition of $G$.

When $G = A_4$, we get a strict $3$-split decomposition with $A = 1$ by taking $B_1 = \{(12)(34), (123), (142),(234)\}$, $B_2 = \{(13)(24), (132), (134), (243)\}$, and $B_3 = \{(14)(23), (124), (143)\}$.

For $G = S_3 \times S_3$, we obtain a strict $3$-split decomposition for $A = A_3 \times A_3$ by taking
$$\begin{array} {ccl}
B_1 & = & \{ ((12),(1)), ((13),(1)), ((23),(1)), ((123),(12)), ((123),(13)), ((123),(23)) \},\\
B_2 & = & \{ ((12),(123)), ((13),(123)), ((23),(123)), ((1),(12)), ((1),(13)), ((1),(23)) \}, \\
B_3 & = & \{ ((12),(132)), ((13),(132)), ((23),(132)), ((132),(12)), ((132),(13)), \\
& & ((132),(23)), ((12),(12)), ((12),(13)),((12),(23)), ((13),(12)), ((13),(13)), \\
& & ((13),(23)), ((23),(12)), ((23),(13)), ((23),(23)) \}.
\end{array}$$

Next, suppose $S$ is a Klein $4$-group and $G = AS$, where $A$ is elementary abelian of order $27$, and $G= AS$ satisfies $A = C_A (s_1) \times C_A (s_2) \times C_A (s_3)$ where $S = \{1, s_1, s_2, s_3 \}$ and $|C_A (s_i)| = 3$ for $i = 1, 2, 3$.  For $i = 1, 2, 3$, let $C_A (s_i) = \langle a_i \rangle$.  Notice that $s_1$ will normalize $C_A (s_2)$ and since $C_A (s_1) \cap C_A (s_2) = 1$, we have that $C_A (s_2) = [C_A (s_2),s_1] \leqslant [A, s_1]$.  Similarly, $C_A (s_3) \leqslant [A,s_1]$.  By Fitting's lemma, $A = C_A (s_1) \times [A,s_1]$, and so, we determine that $[A,s_1] = C_A (s_2) \times C_A (s_3)$.  Similarly, we have $[A,s_2] = C_A (s_1) \times C_A(s_3)$ and $[A,s_3] = C_A (s_1) \times C_A (s_2)$.  By Lemma \ref{lm-eee0}, we have that no two elements in $[A,s_1]s_1$ commute.  Since $a_1$ commutes with $s_1$ and $[A,s_1]$, we see that no two elements in $[A,s_1]a_1 s_1$ will commute and no two elements in $[A,s_1]a_1^2 s_1$ will commute.

We see that $C_G (s_1) = C_A (s_1) S$, so the conjugacy class of $s_1$ has size $|G:C_G (s_1)| = |A:C_A (s_1)| = |[A,s_1]|$.  Notice that $[A,s_1] \langle s_1 \rangle$ is a Frobenius group, so $[A,s_1] s_1 \leqslant {\rm cl} (s_1)$.  By consideration of sizes, we have ${\rm cl} (s_1) = [A,s_1] s_1$.  Similarly, ${\rm cl} (s_2) = [A,s_2] s_2$ and ${\rm cl} (s_3) = [A,s_3] s_3$.  Since $C_A (s_1) \leqslant [A,s_2] \cap [A,s_3]$, we have $C_A (s_1) s_2 \subseteq [A,s_2]s_2$ and $C_A (s_1) s_3 \subseteq [A,s_3] s_3$.  Hence, $C_A (s_1)$ contains no elements in $[A,s_2]a_2 s_2$, $[A,s_2] a_2^2 s_2$, $[A,s_3]a_3 s_3$, and $[A,s_3] a_3^2 a_3$.  Since every element in $[A,s_1]s_1$ is conjugate to $s_1$, we conclude that no element in $[A,s_1]s_1$ commutes with any element in $[A,s_2]a_2 s_2$ or $[A,s_3]a_3 s_3$. Similarly, no element in $[A,s_2]s_2$ commutes with any element in $[A,s_1]a_1s_1$ or $[A,s_3] a_3^2 s_3$ and no element in $[A,s_3]s_3$ commutes with any element in $[A,s_1]a_1^2 s_1$ or $[A,s_2]a_2^2 s_2$.

Notice that $\langle a_1s_1 \rangle \leqslant C_G (a_1s_1)$, so $6 \leqslant |C_G (a_1 s_1)|$.  Since $a_1$ commutes with $s_1$ and $[A,s_1]$, we conclude that all the elements $[A,s_1] a_1 s_1$ are conjugate as are all the elements in $[A,s_1] a_1^2 s_1$.  On the hand, since $a_1 \in C_A (s_1) \leqslant [A,s_2]$, we know that $s_2$ inverts $a_1$, so $(a_1 s_1)^{s_2} = a_1^{s_1} s_2^{s_1} = a_1^2 s_1$.  Hence, $a_1 s_1$ and $a_1^2 s_1$ are conjugate.  This implies that $[A,s_1] a_1 s_1 \cup [A,s_1] a_1^2 s_1 \subseteq {\rm cl} (a_1 s_1)$.  This implies that $18 \leqslant |{\rm cl} (a_1 s_1)|$, and thus, $|C_G (a_1s_1)| \geqslant |G|/|{\rm cl} (a_1s_1)| = 108/18 = 6$.  We deduce that $C_G (a_1 s_1) = \langle a_1 s_1 \rangle$.  Now, $C_G (a_1 s_1)$ contains no elements in $[A,s_3]a_3^2 s_3$.  Using conjugacy, we conclude that no element in $[A,s_1]a_1 s_1$ commutes with any element in $[A,s_3]a_3^2 s_3$.  Similarly, no element in $[A,s_1]a_1^2 s_1$ will commute with any element in $[A,s_2] a_2^2 s_2$, and no element in $[A,s_2] a_2 s_2$ commutes with any element in $[A,s_3]a_3s_3$.  Hence, we have a strict $3$-split decomposition for $G$ with respect to $A$ by taking $$B_1 = [A,s_1] s_1 \cup [A,s_2] a_2 s_2 \cup [A,s_3] a_3 s_3,$$ $$B_2 = [A,s_2] s_2 \cup [A,s_1] a_1 s_1 \cup [A,s_3] a_3^2 s_3,$$ and $$B_3 = [A,s_3] s_3 \cup [A,s_1] a_1^2 s_1 \cup [A,s_2] a_2^2 s_2.$$
The remaining groups were handled in Lemma \ref{3-split-D8}.
\end{proof}

\begin{lm} \label{3-split-even}
Let $G$ be a nonabelian group with a normal abelian subgroup $A$ with even order.  Then $G$ has a strict $3$-split decomposition with respect to $A$ if and only if one of the following holds:
\begin{enumerate}
\item $G$ is a Frobenius group with Frobenius kernel $A$ and a Frobenius complement of order $3$.
\item $A$ has index $2$ in $G$, $|A| \geqslant 6$, and $|Z (G)| = 2$.
\item $G$ is either $D_8$ or $Q_8$ and $A = Z(G)$.
\item $G \cong S_4$ and $A$ is the Klein $4$-subgroup.
\end{enumerate}
\end{lm}

\begin{proof}
We first suppose that $G$ has a strict $3$-split decomposition with respect to $A$.  Suppose $b$ is an element of order $3$ in $G\setminus A$.  By Lemma \ref{3-split prel} (3), $3$ does not divide $|A|$.  Moreover, Lemma \ref{lm-e0} (2) implies that $(o (b) - 1) |C_A(b)| \leqslant 3$.  Since $o(b) - 1=2$, we conclude that $|C_A (b)| = 1$.  This implies that $A \langle b \rangle$ is a Frobenius group.  If $|G:A|$ is odd, then by Lemma \ref{3-split prel} (3), we have $G = A \langle b \rangle$ and conclusion (1) holds.

Let $|G/A|$ be even and $tA$ is an involution in $G/A$ where $t$ is an $2$-element.  Observe that $C_A (t) \langle t \rangle$ is an abelian group.  By Lemma \ref{lm-e0} (1), we have that $|C_A (t)| \leqslant 3$.  Since $|A|$ is even, we see that $C_A (t)$ is not trivial. Thus $|C_A(t)| = 2$.  If $G = A \langle t \rangle$, then $A$ will have index $2$ and $Z (G) = C_A (t)$ will have order $2$.  Since the decomposition is strict, we see that $G \setminus A$ must have three subsets each having size at least $2$, so $|A| = |G \setminus A| \geqslant 6$.  This gives conclusion (2).

We now assume that $|G/A| \geqslant 4$.  Let $T$ be the Sylow $2$-subgroup of $A$.  By Lemma \ref{index2}, we know that $\langle T, t\rangle$ either has order $4$ or is isomorphic to a dihedral, semidihedral, quaternion, or generalized quaternion group. It follows that $T$ is cyclic or $|T| = 4$.

Suppose that $|T| > 2$.  Let $S$ be a Sylow $2$-subgroup of $G$ and observe that $S \cap A = T$.  If $C_S(T)>T$, then there exists $s \in C_S (T) \setminus T$.  Observe that $T \langle s \rangle$ is abelian and not contained in $A$.  Since $|T| > 3$, this would violate Lemma \ref{lm-e0} (1).  Thus, we have $C_S (T) = T$.  Let $T_0$ be a subgroup of order $4$ in $T$. If $C_G(T_0) \neq A$, then for $b \in C_G (T_0) \backslash A$, $T_0\langle b\rangle$ is abelian so it would also violate Lemma \ref{lm-e0} (1).  Thus, $C_G (T_0)=A$.  We know that $G/A$ is isomorphic to a subgroup of the automorphism group of $T_0$.  Since $|G:A| \geqslant 4$, we must have that $T_0$ is a Klein $4$-group and $G/A \cong S_3$.  Observe that $G/A$ acts fixed-point-freely on the Hall $2$-complement $N$ of $A$.  Since $S_3$ is not a Frobenius complement, we have $N = 1$.  Notice that $T_0$ being not cyclic implies that $T = T_0$, and so, $G\cong S_4$ where $A$ is the Klein $4$-subgroup, and $(4)$ holds.

Now, let $|T|=2$. Then $T\leqslant Z(G)$ and in light of the first paragraph, we see that $G/A$ is a $2$-group. If $G/A$ contains an element of order $4$, then $G$ contains an abelian subgroup of order $8$ that is not contained in $A$ and this violates Lemma \ref{lm-e0} (1).  Thus $G/A$  is elementary abelian.  Letting $N$ be the Hall 2-complement of $A$, we see that $G/A$ acts fixed-point-freely on $N$; so $N \neq 1$ would imply that $|G/A|=2$ and conclusion (2) holds.  Thus, we may assume that $N=1$; so $A = T = Z(G)$.  This implies that $G$ is a $2$-group.  We know that if $G$ has order greater than $8$, then $G$ has an abelian subgroup of order at least $8$ and this violates Lemma \ref{lm-e0} (2).  We conclude that $G$ is nonabelian of order $8$ and conclusion (3) holds.

Conversely, notice that each of the groups mentioned has a strict $2$-split decomposition, and so, we obtain a strict $3$-split decomposition by appealing to Lemma \ref{2-3-split}.  The lemma is proved.
\end{proof}

\begin{lm} \label{3-split-nonnormal}
Let  $G$ be a nonabelian group and let $A$ be a nonnormal abelian subgroup of $G$.  Then $G$ has a strict $3$-split decomposition with respect to $A$ if and only if either $G \cong A_4$ and $A$ is a subgroup of order $2$ or $3$ or $G\cong S_4$ and $A$ is a subgroup of order $2$, $3$, or $4$.
\end{lm}
\begin{proof}
Suppose that $G$ has a strict $3$-split decomposition with respect to $A$.  Let $A_1 = A^g \neq A$. Then $|A_1 \cap A|\cdot (|A:A_1 \cap A|-1)\leqslant 3$ by Lemma \ref{lm-e0} (1), and hence one of the following holds: (1) $|A_1 \cap A| = 1$ and $|A:(A_1 \cap A)| \leqslant 4$, (2) $|A_1 \cap A| = 2$ and $|A:A_1 \cap A| \leqslant 2$, or (3) $|A_1 \cap A| = 3$ and $|A : A_1 \cap A| = 2$.  Notice that in case (1) we have $|A| = 2$, $3$, or $4$, in case (2), we have $|A|$ is either $2$ or $4$, and in case (3), we have $|A| = 6$.  If $x$ is any element of $G$ outside $A$, then $o(x) \leqslant 6$ by Lemma \ref{lm-e0} (3).  Note that if $o(x) = 5$, then $\langle x \rangle \cap A = 1$, and applying Lemma \ref{lm-e0} (2) we obtain $o(x) \leqslant 4$ which is a contradiction.  Thus, every nonidentity element of $G$ has order $2$, $3$, $4$, or $6$.  Thus, $|G| = 2^a \cdot 3^b$ for nonnegative integers $a, b$, and hence $G$ is solvable.

By Lemma \ref{lm-e0} (1), we see that any abelian subgroup of $G$ has order at most $6$.  Since $G$ is nonabelian, we see that if $G$ were a $3$-group, then $|G| \geqslant 27$ and $G$ would have an abelian subgroup of order $9$ which is a contradiction.  If $G$ is a $2$-group of order at least $16$, then $G$ would have an abelian subgroup of $8$ which is a contradiction.  This would force $G$ to have order $8$.  By Lemma \ref{lm-e0} (5), we see that $G' = Z(G) \leqslant A$ which implies $A$ is normal, a contradiction.  We conclude that $G$ is neither a $2$-group nor a $3$-group.

Let $Q$ be a Sylow $3$-group of $G$.  We have just shown that $Q$ is nontrivial.  On the other hand, if $|Q| \geqslant 9$, then $G$ will have an abelian subgroup of order $9$, and we have seen that this is a contradiction.  Thus, we have that $|Q| = 3$.  Similarly, if $P$ is a Sylow $2$-subgroup of $G$, then if $|P| \geqslant 16$, then $G$ will have an abelian subgroup of $8$ which is not allowed.  Thus, we have that $|P| \leqslant 8$.  We see that the possibilities for $|G|$ are $6$, $12$, and $24$.  Since $G$ has a strict $3$-split decomposition, there must be at least $6$ elements of $G$ outside of $A$.  This rules out $|G| = 6$.

Notice that if $Q$ is normal in $G$, then $|G:Q C_P (Q)| \leqslant 2$ since $G/Q C_P (Q)$ is isomorphic to a subgroup of ${\rm Aut} (Q) \cong Z_2$.  Since $C_P (Q) > 1$, we have that $Z = Z(P) \cap C_P (Q) > 1$.  Notice that both $P$ and $Q$ will centralize $Z$, so $Z$ is normal in $G$.  By Lemma \ref{lm-e0} (5), we know that $Z \leqslant A$, and since $A$ is not normal, we have $Z < A$.  From the available orders for $A$, this implies that either $A = ZQ$ or $|A| = 4$.  Since $A$ is not normal, we have $A \ne ZQ$, so $|A| = 4$.  Also, $ZQ$ is an abelian subgroup of $G$ and $ZQ \cap A = Z$.  This implies that $(|ZQ:ZQ \cap A| - 1)|ZQ \cap A| = (3-1)2 = 4$ which violates Lemma \ref{lm-e0} (1).  Thus, the case $Q$ is normal cannot occur.

Now, we have that $Q$ is not normal.  Since the number of Sylow $3$-subgroups in $G$ is congruent to $1$ mod $3$, $Q$ must have $4$ conjugates in $G$.  Let $K$ be the kernel of the action of $G$ on the Sylow $3$-subgroups of $G$.  Observe that $G/K$ has order at least $12$ and is isomorphic to a subgroup of $S_4$.  If $K > 1$, then since $|G| \leqslant 24$, we must have $|K| = 2$ and $|G:K| = 12$.  Notice that $K$ will be central in $G$, so $K \leqslant A$.  Notice that $KQ$ is an abelian subgroup.  Replacing $Q$ by a conjugate if necessary, we may assume that $QK \cap A = K$.  We have $(|QK:QK \cap A| - 1)|QK \cap A| = (3-1)2 = 4$ which violates Lemma \ref{lm-e0} (1).  Thus, $K = 1$.
This implies that $G$ is either $A_4$ or $S_4$.  If $G$ is $A_4$, then $A$ can have order $2$ or $3$.  When $G$ is $S_4$, then $A$ can have order $2$, $3$, or $4$.

Note that the $2$-split decompositions for $A_4$ when $A$ is a subgroup of order $2$ in the Klein $4$-subgroup that appeared in Lemma \ref{2-split-nonn} yields strict $3$-split decomposition. The following gives a strict $3$-split decomposition for $A_4$ when $A$ is a Sylow $3$-subgroup:
$$\begin{array}{lcl}
A & = & \{1, (123), (132)\};\\[0.1cm]
B_1 & = & \{(12)(34), (124), (234)\},\\[0.1cm]
B_2 & = & \{(13)(24), (142), (143)\},\\[0.1cm]
B_3 & = & \{(14)(23), (134), (243)\}.\\[0.1cm]
\end{array}$$

Now, we present $3$-split decompositions for $S_4$ with respect to various possible $A$'s.  We present an example for one representative of each of the conjugacy classes of the possible $A$'s.

$$\begin{array}{lcl}
A & = & \{ (1), (12)(34)\}; \\[0.1cm]
B_1 & = & \{ (13)(24), (123), (134), (234), (1243), (1324), (12), (14)\}, \\[0.1cm]
B_2 & = & \{ (14)(23), (132), (142), (143), (1432), (1423), (34), (24)\},\\[0.1cm]
B_3 & = & \{ (1234), (1342), (124), (243), (13), (23) \}.\\[0.1cm]
\hline
A & = & \{ (1), (24)\}; \\[0.1cm]
B_1 & = & \{ (13)(24),(123),(134),(234),(1243),(1324),(12),(14)\}, \\[0.1cm]
B_2 & = & \{ (14)(23), (132), (142), (143), (1432),  (1423), (34)\},\\[0.1cm]
B_3 & = & \{ (12)(34), (1234), (1342),(124),(243), (13),(23) \}.\\[0.1cm]
\hline
A & = & \{ (1), (123), (132)\}; \\[0.1cm]
B_1 & = & \{ (13)(24), (134), (234), (1243), (1324), (12), (14)\}, \\[0.1cm]
B_2 & = & \{ (14)(23), (142), (143), (1432), (1423), (34), (24)\},\\[0.1cm]
B_3 & = & \{ (12)(34), (1234), (1342), (124), (243), (13), (23) \}.\\[0.1cm]
\hline
A & = & \{ (1), (12), (34), (12)(34)\}; \\[0.1cm]
B_1 & = & \{ (13)(24),(123), (134), (234), (1243), (1324), (14) \}, \\[0.1cm]
B_2 & = & \{ (14)(23), (132), (142), (143), (1432), (1423), (24) \},\\[0.1cm]
B_3 & = & \{ (1234), (1342), (124), (243), (13), (23) \}.\\[0.1cm]
\hline
A & = & \{1,  (1234), (13)(24), (1432)\}; \\[0.1cm]
B_1 & = & \{(13), (23), (123), (124), (1243), (1324)\},\\[0.1cm]
B_2 & = & \{(34), (14), (12)(34), (132), (243), (143), (1342)\},\\[0.1cm]
B_3 & = & \{(12), (24), (14)(23), (134), (234), (142), (1423)\}.\\[0.1cm]
\end{array}$$
The lemma is proved.
\end{proof}

\section{Decompositions of $L_2 (q)$, ${\rm Sz} (q)$, and ${\rm PGL}_2 (q)$}
Let $G$ be a nonabelian group and $\min (G)$ the minimal number $n$ for which the group $G$ has a strict $n$-split decomposition for some abelian group $A$.  Clearly, if $H$ and $K$ are two maximal abelian subgroups of $G$ with $H \cap K = 1$, then $\min(G) \geqslant \min \{ |H|, |K|\}-1$. In particular, if $G$ has a maximal abelian subgroup $A$ which intersects trivially with some of its conjugates, then $\min (G) \geqslant |A| - 1$. We now investigate $\min (G)$ for two families of groups.

In what follows, we restrict our attention to the almost simple groups (recall that $G$ is almost simple if $S \leqslant G \leqslant {\rm Aut}(S)$ for some nonabelian simple group $S$).  We begin with the following result:

\begin{lm} \label{l7-1}
Suppose that $q = p^m \geqslant 4$, with $p$ a prime and $m \geqslant 1$ an integer. Then the following assertions hold.
\begin{itemize}
\item[$(1)$] If $p=2$ and $G = L_2 (q) \cong {\rm PGL}(2,q)$, then $\min (G) = q$.
\item[$(2)$] When $p$ is an odd prime, then $\min (G) = q - 1$ if $G = L_2(q)$, and  $\min (G) = q$, if $G = {\rm PGL}(2, q)$.
\end{itemize}
In particular, we have the following:
\begin{itemize}
\item If $S \cong A_5 \cong L_2(4) \cong L_2(5)$,  then $\min (S) = 4$.
\item If $S \cong L_3 (2) \cong L_2(7)$, then $\min (S) = 6$.
\item If $S \cong A_6 \cong L_2(9)$, then $\min (S) = 8$.
\item If $S \cong S_5 \cong {\rm PGL}_2(5)$, then $\min (S) = 5$.
\end{itemize}
\end{lm}

\begin{proof}
It is known that $G$ contains abelian subgroups $C$, $D$, $F$, of orders $(q-1)/k$, $q$, and $(q+1)/k$, respectively where $k = 1$ if either $p = 2$ or $G = {\rm PGL} (2,q)$ when $p$ is odd and $k = 2$ if $G = L_2(q)$ when $p$ is odd; and every two distinct conjugates of any of these groups intersect trivially.  Furthermore, every element of $G$ is a conjugate of an element in $C \cup D\cup F$  (see Satz II.8.5 of \cite {hup}).  Label the elements $C = \{ 1, c_1, \ldots, c_{(q-1)/k - 1}\}$, $D = \{ 1, d_1, \ldots, d_{q-1}\}$, and $F = \{ 1, f_1, \ldots, f_{(q+1)/k - 1}\}$.  Let $\{a_1, \dots, a_r \}$ be a transversal for $N_C = N_G (C)$ in $G$, let $\{b_1, \dots, b_s \}$ be a transversal for $N_D = N_G (D)$ in $G$, and let $\{ g_1, \dots, g_t \}$ be a transversal for $N_F = N_G (F)$ in $G$.  Note that each of $C$, $D$, and $F$ are the centralizers for the nonidentity element they contain.  So the only way two nonidentity elements of $G$ can commute is if they lie in the same conjugate of one of these three subgroups.
Now, we can define the sets
\begin{itemize}
\item[] $R_i = \{c_i^{a_1}, \ldots, c_i^{a_r}\}$, $i=1, 2, \ldots, (q-1)/k - 1$,
\item[] $S_j = \{d_j^{b_1}, \ldots, d_j^{b_s}\}$, $j=1, 2, \ldots, q-1$,
\item[] $T_k = \{f_k^{g_1}, \ldots, f_k^{g_t}\}$, $k=1, 2, \ldots, (q+1)/k - 1$.
\end{itemize}
We take $B_i = R_i \cup S_i \cup T_i$ when $1 \leqslant i\leqslant (q-1)/k - 1$.  Note that $(q-1)/k = (q+1)/k - 1$ for both choices of $k$.  We take $B_{(q-1)/k} = S_{(q-1)/k} \cup T_{(q-1)/k}$.  Finally, we take  $B_j = T_j$ where $j = q$ if $p$ is even and $(q+1)/k + 1 \leqslant j \leqslant q-1$ when $p$ is odd.  Set $A = F$.  Then $G = A \uplus B_1 \uplus \cdots \uplus B_q$ is a strict $q$-split decomposition of $G$ when $p$ is $2$ or when $G = {\rm PGL}_2 (q)$ when $p$ is odd and $G = A \uplus B_1 \uplus \cdots \uplus B_{q-1}$ is a strict $(q-1)$-split decomposition when $G = L_2(q)$ when $p$ is odd.

Suppose now that $G$ has a strict $u$-split decomposition: $G = A' \uplus B_1' \uplus \cdots \uplus B_u'$.  Then we can assume that $A' \cap F = 1$ and hence $u + 1 \geqslant |F:A' \cap F| = |F| =  (q + 1)/k$.  When $p$ is even or when $G = {\rm PGL}_2 (q)$, we conclude that $u \geqslant q$, and so, $\min (G) = q$.  In a similar fashion, we may assume that $A' \cap D = 1$, and hence, $u + 1 \geqslant |D| = q$.  We deduce that $u \geqslant q - 1$. Therefore,  we obtain $\min (G) = q-1$ when $G = L_2(q)$ and $p$ is odd.
\end{proof}

\begin{lm} \label{l7-2}
Let $S$ be a simple group with $\min (S) = n$. Then $|S| \leqslant (n+1)^{n+1}$.  In particular, $A_5 \cong L_2(4)\cong L_2(5)$ is the only  simple group $S$ with $\min (S) \leqslant 4$.
\end{lm}
\begin{proof}
Let $S=A\uplus B_1 \uplus \cdots \uplus B_n$ be a strict $n$-split decomposition of $S$, where $n=\min (S)$.  From \cite{Vdovin} one can deduce that $|S| \leqslant |B|^{|B|}$ where $B$ is an abelian subgroup of maximal order. By  \cite{Zenkov}, there exists $x\in S$ such that $A\cap B^x=1$ and hence $|B|\leqslant n+1$. This gives the desired estimate.

Suppose that  ${\rm min}(S) \leqslant 4$.  By  the previous paragraph  $|S|\leqslant 5^5=3125$, and hence $S$ is isomorphic to $A_7$  or $L_2(q)$, $q=4, 8, 9, 11, 13, 17$. By Lemma \ref{l7-1} (1) and (2), $\min (L_2(4)) = 4$, while $\min(L_2(q)) \geqslant 8$, for  $q=8, 9, 11, 13, 17$. Finally, if $S\cong A_7$, then $S$ has a maximal abelian subgroup of order $7$, which intersects trivially with some of its conjugates, and hence $\min (S)\geqslant 6$.
This completes the proof.
\end{proof}

The Suzuki groups ${\rm Sz}(q)$, an infinite series of simple groups of Lie type, were defined in \cite{Suzuki2, Suzuki} as subgroups of the groups  $L_4(q)$ with $q = 2^{2n+1}$ elements and set $r= 2^{n+1}$.

By \cite[Theorem 7]{Suzuki}, the order of ${\rm Sz}(q)$ is $$|{\rm Sz}(q)| = q^2(q-1)(q^2+1)= q^2(q-1)(q+r+1)(q-r+1),$$  note that  these factors are mutually coprime.
We are now ready to find a strict $k$-split decomposition for the Suzuki groups.

\begin{lm} \label{l7-3}
If $G={\rm Sz}(q)$, where $q=2^{2n+1}\geqslant 8$, then ${\rm min}(G)=2q-1$.
\end{lm}

\begin{proof}
By Lemma XI.11.6 of \cite{hupIII}, $G$ is partitioned by its Sylow $2$-subgroups and its cyclic subgroups of order $q-1$, $q-r+1$, and $q+r+1$.  Looking at the proof of Lemma XI.11.6, we see that the cyclic subgroups of order $q-1$, $q-r+1$, and $q+r+1$ are the centralizers of their nonidentity elements.  We obtain the sets $C_1, \dots, C_{q-2}$ so that each $C_i$ contains one nonidentity element from each of the cyclic subgroups of order $q-1$, the sets $D_1, \dots, D_{q-r}$ so that each $D_i$ contains one nonidentity element from each of the the cyclic subgroups of order $q-r+1$, and the sets $E_1, \dots, E_{q+r}$ so that each $E_i$ contains one nonidentity element from each of the the cyclic subgroups of order $q+r$.  Note that $r < q-1$, so $q+r < 2q -1$.

Let $P$ be a Sylow $2$-subgroup of $G$.  By Theorem VIII.7.9 of \cite{hupII} and Lemma XI.11.2 of \cite{hupIII}, $Z(P)$ is an elementary abelian $2$-group of order $q=2^{2n+1}$ and every element outside $Z(P)$ has order $4$.  Observe that $P$ is the centralizer in $G$ of all of the nontrivial elements of $Z(P)$.  Label elements in $Z(P) = \{ z_0 = 1, z_1, \dots, z_{q-1} \}$. If $x \in P \setminus Z(P)$, then $P_0 = \langle Z(P), x \rangle \leqslant C_G (x)$.  In the proof of Lemma XI.11.7 of \cite{hupIII}, we see that the elements of order $4$ in $G$ lie in two conjugacy classes.  It follows that $|C_G (x)| = 2 |Z(P)|$.  This implies that $C_G (x) = P_0$.  It follows that if $\{ x_1, \dots, x_q \}$ is a transversal for $Z(P)$ in $P$, then $C_G (x_i z_j) = \langle Z(P),x_i \rangle$ for $i = 1,\dots, q$ and $j = 1, \dots, q-1 $.  Let $\{ a_1, \dots, a_t \}$ for a transversal for $N_G (P)$ in $G$.  We define $F_i$ as follows. For $1 \leqslant i \leqslant q-1$, set $F_i = \{ z_i^{a_1}, \dots, z_i^{a_t} \}$.  For $q \leqslant i \leqslant 2q - 1$, we define $$F_i = \bigcup_{j=1}^q \left\{ (x_j z_{i-q})^{a_1}, \dots, (x_j z_{i-q})^{a_r}\right\}.$$
Finally, we define the $B_i$'s.  For $1 \leqslant i \leqslant q-r$, set $B_i = C_i \cup D_i \cup E_i \cup F_i$; for $q-r+1\leqslant i \leqslant q-2$, set $B_i = C_i \cup E_i \cup F_i$; for $q-1 \leqslant i \leqslant q+r$, set $B_i = E_i \cup F_i$; and for $q+r+1 \leqslant i \leqslant 2q-1$, set $B_i= F_i$.  Take $A = 1$.  Then $G = A \uplus B_1 \uplus \cdots \uplus B_{2q-1}$ is a strict $(2q-1)$-split decomposition of $G$.

Note that $G$ has a maximal abelian subgroup $A$ of order $2q$ and $A$  has a conjugate that it intersects trivially, so we know that $\min (G) \geqslant 2q - 1$.  This gives the conclusion that $\min (G) = 2q - 1$.
\end{proof}


\section{Some upper bounds}
Suppose $G$ has an $n$-split decomposition with respect to an abelian subgroup $A$.  In what follows, we  show that the index $|G:A|$  is bounded by some function of $n$.  Since for every positive integer $n$ we can find Frobenius groups with arbitrarily large abelian Frobenius kernels whose Frobenius complements are cyclic of order $n+1$, we can use Theorem \ref{Frobenius} to see that it is not possible to bound $|G|$, particularly $|A|$, in terms of $n$.  However, we now show that we can bound the index $|G:A|$ in terms of $n$.  We have not worked to obtain optimal bounds, and in fact, we are sure that the bounds  obtained are far from optimal.  We see from the first couple of paragraphs of the proof that when $A$ is not normal in $G$, then it is possible to bound $|G|$ in terms of $n$.

\begin{theorem}
There exists a positive integer valued function $f$ defined on the positive integers so that if $G$ has an $n$-split decomposition with respect an abelian subgroup $A$, then $|G:A|$ is bounded by $f (n)$.
\end{theorem}
\begin{proof}
Notice that we have the result when $n=2$ and $3$.  Thus, we may assume that $n \geqslant 4$.  We begin by noting that it suffices to prove that the size of all the abelian subgroups of $G$ are bounded by a function of $n$. In fact, if all of the abelian subgroups of $G$ have order at most $m$, then $|G| \leqslant m!$, see Problem 1D.11 in \cite{isaacs}.

Suppose first that $A$ is not normal.  Let $U$ be an abelian subgroup of $G$ that is not contained in $A$.  We know that $|U| \leqslant 2n$ by Lemma \ref{lm-e0} (1).  Also, since $A$ is not normal, there is some conjugate of $A$ that does not contain $A$.  The above work shows that the size of the conjugate is bounded by $2n$, and so $|A|$ is bounded by $2n$.  Therefore, we conclude that the size of every abelian subgroup of $G$ is bounded by $2n$, and we see that $|G| \leqslant (2n)!$.  Since $|G:A| \leqslant |G|$, this gives the result

We now assume that $A$ is normal in $G$.
If $U$ is an abelian subgroup of $G$ that is not contained in  $A$, then $|U: U \cap A| \leqslant n + 1$ and $|U| \leqslant 2n$ by Lemma \ref{lm-e0}.   If $x$ lies in $G \setminus A$, then $o(Ax) \leqslant n + 1$ by Lemma \ref{lm-e0}.  Thus, the only primes that can divide $|G:A|$ must be less than or equal to $n + 1$.  In particular, the number of distinct prime divisors of $|G:A|$ is at most the number of primes less than or equal to $n + 1$ which is certainly bounded by $n$.

Let $P$ be a Sylow $p$-subgroup of $G$.  It suffices to show that $|PA:A|$ is bounded in terms of $n$.  Suppose $p$ does not divide $|A|$.  If $U$ is an abelian subgroup of $P$, then we have $U \cap A = 1$, so $|U| \leqslant n+1$.   Applying the observation in the first paragraph of the proof, we have $|P| \leqslant (n+1)!$.

Thus, we may assume that $p$ divides $|A|$.  We see that $P$ will be a $p$-group that has the $k$-split decomposition $(A \cap P) \uplus (B_1 \cap P) \uplus \cdots \uplus (B_n \cap P)$ where $k$ is the number of the sets $B_i \cap P$ that are not empty. Note that we have not required that the decomposition be strict, so sets of size one are allowed.  We will prove that $|P:P \cap A| \leqslant (n^2)!$.  If $k < n$, then $(k^2)! \leqslant (n^2)!$, so working by induction on $n$, we may assume that $k = n$.

For the rest of this proof, we assume that $G$ is a $p$-group for some prime $p \leqslant n$.  Suppose $B$ is an abelian normal subgroup of $G$ that is not contained in $A$ and $C_G(B) = B$.  Then as above $|B| \leqslant 2n$.  Also, by the $N/C$-theorem, we have that $$G/B =N_G(B)/C_G (B) \leqslant {\rm Aut} (B) \leqslant {\rm Sym} (B),$$
so $|G:B| \leqslant (2n)!$.  It follows that $$|G:A| < |G| \leqslant (2n) (2n)! \leqslant (n^2)!,$$ which yields the desired conclusion.

Suppose $U$ is an abelian normal subgroup of $G$ that is not contained in $A$.  We claim that there exists a subgroup $B$ normal in $G$ so that $U \leqslant B$, $B$ is abelian, and $B = C_G(B)$.   Observe that $U \leqslant C_G(U)$ and $C_G(U)$ is normal in $G$.  If $U = C_G(U)$, then take $B = U$, and we are done.  Thus, we may assume $U < C_G (U)$. Since $G$ is a $p$-group, we can find $V$ normal in $G$ so that $U < V \leqslant C_G(U)$ and $|V:U| = p$.  Notice that $U$ is central in $V$ and $V/U$ has order $p$, so $V/U$ is cyclic.  This implies that $V$ is abelian.  Also, $C_G (V) \leqslant C_G(U)$, so $|C_G (V):V| < |C_G (U):U|$.  Working by induction, on $|C_G (U):U|$, we obtain the conclusion.  Using the existence of $B$ and the previous paragraph, we see that $|G:A| \leqslant (n^2)!$.

Thus, we may assume that $A$ contains every normal abelian subgroup of $G$ and that $A = C_G (A)$.  Suppose $U$ is a subgroup of $A$ that is maximal such that it is normal in $G$ and $C_G(U)$ is not contained in $A$.  Notice that such a subgroup $U > 1$ must exist since $1 < Z(G) \leqslant A$ and $G = C_G (Z(G))$ is not contained in $A$.  Thus, there is an element $g$ in $G \setminus A$ that centralizes $U$.  Thus, $U \langle g\rangle$ is an abelian subgroup of $G$ that is not contained in $A$.  Observe that $U = U \langle g \rangle \cap A$, and by Lemma \ref{lm-e0} (1), we know that $|U| =  |U\langle g\rangle \cap A| \leqslant n$.  Since $A = C_G (A)$, we know that $U < A$.  Hence, we can find $V$ so that $U < V \leqslant A$, $V$ is normal in $G$, and $|V:U| = p$.  Since $p \leqslant n$, we have $|V| \leqslant n^2$.  Now, $G/C_G (V)$ is a subgroup of ${\rm Aut} (V)$, so $|G:C_G (V)|$ is bounded by $(n^2)!$.  Now, the choice of $V$ implies that $C_G (V) \leqslant A$ and since $V \leqslant A$ and $A$ is abelian, we have $C_G (V) = A$.  Thus, we now have that $|G:A| \leqslant (n^2)!$.
\end{proof}

We do have some cases where we can obtain a better bound.

\begin{lm}
Suppose $G$ has an $n$-split decomposition with respect to $A = 1$, then $|G| \leqslant n!$.
\end{lm}

\begin{proof}
Let $p$ be a prime divisor of $|G|$, and let $P\in {\rm Syl}_p(G)$. Let $M$ be a maximal abelian normal
subgroup of $P$ and set $m = |M|$. Then $M = C_P(M)$.  Thus, $P/M$ is isomorphic to a subgroup of ${\rm Aut} (M) \leqslant S_m$, which forces $|P:M|$ to divide $m!$.  Since $m\leqslant n$, we conclude that $|P|$ divides the $p$-part of $n!$.  The claim follows from the fact that $|G|$ is the least common multiple of the orders of
Sylow subgroups of $G$.
\end{proof}

\begin{lm}
If $G$ has an $n$-split decomposition with respect to $A = Z(G)$ where $|A|=n$, then $|G:A| \leqslant 2n$.
\end{lm}

\begin{proof}
Let $k = |G:A|$.  As the cosets of the center are commuting subsets of $G$, no $B_i$ can contain more than one element of any coset of $A$, and it cannot contain an element of $A$.  We see that $|B_i| \leqslant k-1$.  On the one hand,  we have $|G| = |G:A||A| = kn$, and on the other hand, we obtain  $$|G| = |A| + \sum_{i=1}^n |B_i| \leqslant n + (k-1)n = kn.$$  We must have equality throughout this inequality, so  $|B_i|=k-1$ for all $i$.  Thus each $B_i$ contains representatives of every nontrivial coset of $A$.  This shows that the complementary graph of $\Delta(G)$, which is called noncommuting graph  of $G$ and denoted by $\nabla(G)$, is a complete $(k-1)$-partite graph.  Now, by Proposition 3(ii) in \cite{akm}, $G/A$ is an elementary abelian $2$-group and the size of the class of $g$, for every noncentral element $g\in G$, is $k/2$.  Finally, since $G/A$ is abelian, the entire conjugacy class of $g$ is contained in the coset $Ag$, which has size $n$. Thus $k/2\leqslant n$, so $k\leqslant 2n$.
\end{proof}

{\em Some examples.} Both $D_8$ and $Q_8$ have $2$-split decompositions with respect to their centers who have order $2$.  There are also many examples where $n$ is large.  One family of examples is the Suzuki $2$-groups (see \cite{Higman}).  We omit the details here.

\begin{center}
 {\bf Acknowledgments }
\end{center}

This work was done during the fourth author had a visiting position at the Department of Mathematical Sciences, Kent State University, USA. He would like to thank the hospitality of the Department of Mathematical Sciences of KSU.  He would like also to express his sincere thanks to Prof.  I. M. Isaacs for helpful comments.  The work of D. V. Lytkina and V. D. Mazurov is supported by grant 17-51-560005 of Russian Foundation of Basic Research.

We would like to thank the referee for a thorough reading of this paper and the resulting helpful suggestions.

\noindent {\sc M. L. Lewis}\\[0.2cm]
{\sc Department of Mathematical Sciences, Kent State
University,}\\ {\sc  Kent, Ohio $44242$, United States of
America}\\[0.1cm]
{\em E-mail address}: {\tt  lewis@math.kent.edu}\\[0.3cm]
{\sc D. V. Lytkina}\\[0.2cm]
{\sc Siberian State University of Telecommunications and Information Sciences,
 Novosibirsk State University, Novosibirsk, Russia}\\[0.1cm]
{\em E-mail address}: {\tt daria.lytkin@gmail.com}\\[0.3cm]
 {\sc V. D. Mazurov}\\[0.2cm]
{\sc Sobolev Institute of Mathematics and Novosibirsk State University,
 Novosibirsk, Russia},\\[0.1cm]
{\em E-mail address:} {\tt mazurov@math.nsc.ru}\\[0.3cm]
 {\sc A. R. Moghaddamfar}\\[0.2cm]
{\sc Faculty of Mathematics, K. N. Toosi
University of Technology,
 P. O. Box $16315$--$1618$, Tehran, Iran,}\\[0.1cm]
 {\sc and}\\[0.1cm]
 {\sc Department of Mathematical Sciences, Kent State
University,}\\ {\sc  Kent, Ohio $44242$, United States of
America}\\[0.1cm]
{\em E-mail addresses:}:  {\tt
moghadam@kntu.ac.ir}, and {\tt amoghadd@kent.edu}\\[0.3cm]

\begin{thebibliography}{99}
\bibitem{akm}  M. Akbari and A. R. Moghaddamfar, The existence or nonexistence of non-commuting graphs with particular properties, {\em J. Algebra Appl.}, 13 (1) (2014), 1350064, 11 pp.

\bibitem{ma}  M. Akbari and A. R. Moghaddamfar, Groups for which the noncommuting graph is a split graph, {\em Int. J. Group Theory}, 6 (1) (2017), 29--35.

\bibitem{Higman} G. Higman, Suzuki $2$-groups, {\em  Illinois J. Math.}, 7 (1963), 79--96.

\bibitem{hup} B. Huppert, ``Endliche Gruppen I,'' Springer-Verlag, Berlin, 1983.

\bibitem{hupII} B. Huppert and N. Blackburn, ``Finite Groups II,'' Springer-Verlag, Berlin, 1982.

\bibitem{hupIII} B. Huppert and N. Blackburn, ``Finite Groups III,'' Springer-Verlag, Berlin, 1982.

\bibitem{isaacs}  I. M. Isaacs, {\em Finite Group Theory}, Graduate Studies in Mathematics, 92. American Mathematical Society, Providence, RI, 2008.

\bibitem{mszz} A. R. Moghaddamfar, W. J. Shi, W. Zhou and A. R. Zokayi, On the noncommuting graph associated with a finite group, {\em  Siberian Math. J.}, 46 (2) (2005), 325--332.

\bibitem{Suzuki2} M. Suzuki, A new type of simple groups of finite order, {\em  Proc. Nat. Acad. Sci. U. S. A.}, 46 (1960) 868--870.

\bibitem{Suzuki} M. Suzuki, On a class of doubly transitive groups, {\em Ann. of Math.}, 75 (1) (1962), 105--145.

\bibitem{Vdovin} E. P. Vdovin, Maximal orders of abelian subgroups in finite simple groups, {\em Algebra and Logic}, 38 (2) (1999), 67--83.

\bibitem{Zenkov} V. I. Zenkov, Intersections of abelian subgroups in finite groups, {\em Math. Notes}, 56 (1-2) (1994), 869--871.

\end{thebibliography}
\end{document}